\documentclass[12pt, a4paper]{amsart}
\usepackage{amsfonts}
\usepackage{amsthm}
\usepackage{amssymb}
\usepackage{amsmath}
\usepackage{enumerate}
\usepackage{url}
\usepackage{epsfig}
\usepackage{caption}
\usepackage{subcaption}
\usepackage{enumitem}

\usepackage{color}
\usepackage{comment}
\usepackage{a4wide}

\allowdisplaybreaks

\numberwithin{equation}{section}

\begin{document}
	
\title{On the order of magnitude of Sudler products  II}
\author{Sigrid Grepstad}
\address{Department of Mathematical Sciences, 
	Norwegian University of Science and Technology,
	NO-7491 Trondheim, Norway;
    E-mail address:  sigrid.grepstad@ntnu.no }
\author{Mario Neum\"uller}
\address{Johannes Kepler Universit\"at, 
      Linz, Austria;
	E-mail address: mario.neumueller@jku.at }
\author{Agamemnon Zafeiropoulos}
\address{Department of Mathematical Sciences, 
	Norwegian University of Science and Technology,
	NO-7491 Trondheim, Norway; 
	E-mail address: agamemnon.zafeiropoulos@ntnu.no}
\thanks{SG and AZ are supported by Grant 275113 of the Research Council of Norway.}\thanks{MN is funded by FWF projects F5509-N26,  F5512-N26 and P29910-N35.}

\begin{abstract}  
We study the asymptotic behavior of Sudler products $P_N(\alpha)= \prod_{r=1}^{N}2|\sin \pi r\alpha|$ for quadratic irrationals $\alpha \in \mathbb{R}$. In particular, we verify the convergence of certain perturbed Sudler products along subsequences, and show that $\liminf_N P_N(\alpha) = 0$ and $\limsup_N P_N(\alpha)/N = \infty$ whenever the maximal digit in the continued fraction expansion of $\alpha$ exceeds $23$. This generalizes known results for the period one case $\alpha=[0; \overline{a}]$.
\end{abstract}


\maketitle

\newtheorem{prop}{Proposition}
\newtheorem{claim}{Claim}
\newtheorem{thm}{Theorem}
\newtheorem*{thm*}{Theorem}
\newtheorem{thma}{Theorem A}
\newtheorem{defn}{Definition}
\newtheorem{conj}{Conjecture}
\newtheorem{cor}{Corollary}
\newtheorem*{cor*}{Corollary}
\newtheorem{lem}{Lemma}

\theoremstyle{definition}
\newtheorem{exmp}{Example}
\newtheorem*{ques*}{Questions}

\theoremstyle{remark}
\newtheorem{rem}{Remark}
\newtheorem*{rem*}{Remark}
\newtheorem*{remMario}{\color[rgb]{0.2,0.8,0.2}{Remark (Mario)}}

\newcommand{\bfa}{\bf{a}}
\newcommand{\be}{\beta}
\newcommand {\ve} {\varepsilon}
\newcommand{\RR}{\mathbb{R}}
\newcommand{\QQ}{\mathbb{Q}}
\newcommand{\ZZ}{\mathbb{Z}}
\newcommand{\calO}{\mathcal{O}}
\newcommand{\ql}{q_{\ell}(\alpha_{\sigma_k})}
\newcommand{\qlplusone}{q_{\ell+1}(\alpha_{\sigma_k})}
\newcommand{\pl}{p_{\ell}(\alpha_{\sigma_k})}
\newcommand{\qlplustau}{q_{\ell+1}(\alpha_{\tau_k})}
\newcommand{\qltau}{q_{\ell}(\alpha_{\tau_k})}
\newcommand{\pltau}{p_{\ell}(\alpha_{\tau_k})}
\newcommand{\ckek}{|c_ke_k|}
\newcommand{\cu}{\mathcal{U}}

\section{Introduction and Main Results}

\subsection{Introduction} Let $\alpha\in\mathbb{R}$ and $N\geq 1$ be an integer. The Sudler product at stage $N$ and with parameter $\alpha$ is defined as
\begin{equation*}
	P_N(\alpha)=\prod_{r=1}^N|2\sin\pi r\alpha |.
\end{equation*}
Sudler products have been studied extensively, as they bear connections with several areas of research; we mention partition theory, Pad\'e approximants and dynamical systems, and refer to  \cite{mv} and references therein for further examples and details. In the present paper our main focus will be on the asymptotic order of magnitude of $P_N(\alpha)$. This topic has received much attention in recent years, and we begin by briefly reviewing key results relevant to the main results of this paper. For a more detailed overview of the asymptotic behavior of $P_N(\alpha)$, we refer to the survey paper \cite{sigrid3}.

Erd\H{o}s and Szekeres showed  in \cite{erdos} that $\liminf_{N\to\infty}P_N(\alpha)=0$ for almost all $\alpha$, and conjectured that this result is true for \emph{all} values of $\alpha$.  Lubinsky \cite{lubinsky} later confirmed that $\liminf_{N\to \infty} P_N(\alpha)=0$ whenever $\alpha$ has unbounded partial quotients in its continued fraction expansion $\alpha=[0;a_1,a_2,\ldots]$. 
 
 More recently, Mestel and Verschueren \cite{mv} studied the behavior of Sudler products $P_N(\phi)$, where $\phi = [0;1,1,\ldots]$ is the fractional part of the golden ratio. Their precise result was the following.    
 
\textbf{Theorem (Mestel, Verschueren):} \textit{Let $\phi =(\sqrt{5}-1)/2$ be the fractional part of the golden ratio and $(F_n)_{n=0}^{\infty}$ be the sequence of Fibonacci numbers. Then  there exists a constant $C>0$ such that 
\begin{equation*} 
   \lim_{n\to\infty} P_{F_n}(\phi) =C.
 	\end{equation*}
 	Moreover, for the same constant $C$ we have $\lim\limits_{n\to \infty} \dfrac{P_{F_n-1}(\phi)}{F_{n}} = \dfrac{C\sqrt{5}}{2\pi} \cdot$}
 
Here the appearance of the Fibonacci sequence is not at all surprising, as it is the sequence of denominators associated with the continued fraction expansion of $\phi$. The proof of the result relies on the specific continued fraction expansion $\phi = [0;1,1,\ldots]$ and the algebraic properties of the sequence $(F_n)_{n=0}^{\infty}$.  We now know that the convergence property for $P_{F_n}(\phi)$ is a special case of a phenomenon exhibited by all quadratic irrationals \cite{GN18}.

\textbf{Theorem (Grepstad, Neum\"uller):} \textit{Let $\alpha = [0;\overline{a_1, a_2, \ldots, a_{\ell}}] $ be a purely periodic quadratic irrational, where $\ell\geq 1$ and $a_1,\ldots, a_{\ell} \in \mathbb{N}$, and let $(q_n)_{n=1}^{\infty}$ be the sequence of denominators of convergents of $\alpha$. Then there exist constants $ C_1, C_2, \ldots, C_{\ell}>0$ such that 
\begin{equation} \label{constants} 
 \lim_{m\to\infty}P_{q_{m\ell+k}}(\alpha) = C_k, \qquad k= 1, 2, \ldots, \ell. 
 \end{equation}  
 Moreover if $\beta= [0;b_1,\ldots,b_h,\overline{a_{1},\ldots, a_{\ell}}]$ is a quadratic irrational with the same periodic part as $\alpha$ in its continued fraction expansion, then 
$$ \lim_{m\to\infty}P_{q_{h+m\ell + k}}(\beta) = C_k, \qquad k=1,2,\ldots, \ell.   $$
for the same constants $C_1,\ldots, C_{\ell}$. }

Later on, Grepstad, Kaltenb\"ock and Neum\"uller employed the factorisation technique used in the proof of Mestel and Verschueren's result to show that $\liminf P_N(\phi)>0$ \cite{gkn}, finally disproving the conjecture in \cite{erdos}. 

The proof of the lower bound on $P_N(\phi)$ given in \cite{gkn} involved studying a perturbed Sudler product $\prod_{r=1}^{N}2|\sin \pi( r\phi + \varepsilon)|.$ A systematic treatment of such perturbed products was conducted in \cite{ATZ20}, where the result by Mestel and Verschueren was generalized to quadratic irrationals of the form $\be=[0;b,b,\ldots]$ by an in-depth study of the product 
\begin{equation} \label{perturbed}
 P_{q_n}(\be, \ve) = \prod_{r=1}^{q_n}2\Big|\sin \pi\Big(r\be +(-1)^n \frac{\ve}{q_n}\Big)\!\Big| .
\end{equation} 
In \cite{ATZ20} it is shown that for each digit $b\geq 1$, the sequence of functions $P_{q_n}(\be,\ve)$ converges locally uniformly to an explicitly defined function $G_b(\ve)$, and from this the authors deduce the following strong result on the asymptotic behavior of $P_N(\be)$.

\textbf{Theorem (Aistleitner, Technau, Zafeiropoulos):} \textit{Let $\be = [0;b,b,\ldots]$, where $b\geq 1$.  The following holds. \vspace{-3mm}
\begin{itemize}
	\item[(i)]  If $b \leq 5$, then $\liminf \limits_{N \to \infty} P_N(\be) > 0$ and $\limsup \limits_{N\to\infty}\dfrac{P_N(\be)}{N} < \infty.$ \vspace{1mm}
	\item[(ii)] If $b \geq 6$, then $\liminf \limits_{N \to \infty} P_N(\be) = 0$ and $\limsup \limits_{N\to\infty}\dfrac{P_N(\be)}{N} = \infty.$
\end{itemize} }

 The theorem above gives a complete description of the asymptotic order of magnitude of $P_N(\beta)$ for irrationals $\beta=[0;b,b,\ldots]$. The main objective of this paper is to study the asymptotic behaviour of $P_N(\beta)$ for \emph{arbitrary} quadratic irrationals $\beta$, that is irrationals whose continued fraction expansions are eventually periodic with some period length $\ell$. It turns out that for such $\beta$, the sequence of functions $P_{q_n}(\beta, \varepsilon)$ defined in \eqref{perturbed} will converge along specific subsequences to $\ell$ explicitly defined functions $G_{k}(\beta,\varepsilon)$, $1 \leq k \leq \ell$ (see Theorem \ref{thm:limitFunc} below). This is, in some sense, the expected generalization of the result on $\beta=[0;b,b,\ldots]$ in \cite{ATZ20}. As a consequence, we obtain a partial extension of the theorem above to arbitrary quadratic irrationals (Theorem \ref{thm:ak23}).
 
As we shall explain later, 
 the asymptotic behaviour of $P_N(\beta)$ for $\beta= [0;b_1,\ldots,b_h,\overline{a_{1},\ldots, a_{\ell}}]$ is similar to that of $P_N(\alpha)$, where $\alpha = [0;\overline{a_1,\ldots, a_\ell}]$. It turns out that purely periodic irrationals are quite easier to analyse in terms of their continued fraction expansions. Moreover, certain relations following from such an analysis are needed in the statement of our main results. Let us therefore briefly review certain basic properties for the convergents of purely periodic irrationals.


\subsection{The irrational $\alpha = [0; \overline{a_1,\ldots, a_{\ell}}]$\label{subsec:contfracprelim}}  
The $n$-th convergent of $\alpha$ is the number $p_n/q_n$, where 
\begin{eqnarray*}	
p_{n+1} &=& a_n p_n + p_{n-1}, \qquad p_0=1, \quad p_1 = 0 , \\
q_{n+1} &=& a_n q_n + q_{n-1}, \qquad q_0=0, \quad q_1 = 1 .
\end{eqnarray*}
We mention that  the dependence of $p_n$ and $q_n$ on $\alpha$ is not explicitly stated, but if necessary we will write $p_n(\alpha)$ and $q_n(\alpha)$ to make this dependence explicit. The sequence of convergents satisfies 
\begin{equation*} 
\frac{p_1}{q_1} < \frac{p_3}{q_3} < \cdots < \alpha < \cdots < \frac{p_4}{q_4} < \frac{p_2}{q_2} \, 
\end{equation*}
and
\begin{equation*}
 \frac{1}{2q_{n+1}q_n} \, < \, \left| \alpha - \frac{p_n}{q_n} \right| \, < \, \frac{1}{q_{n+1}q_n}, \quad n\geq 1 .
\end{equation*}
We use the notation of \cite{GN18} and set 
\begin{align} 
&c(\alpha) = c =q_{\ell+1}+p_{\ell}, \nonumber  \\
&a(\alpha)= a =\frac{c(\alpha)+\sqrt{c(\alpha)^2+4(-1)^{\ell-1}}}{2},  \label{abc}  \\
&b(\alpha)= b =\frac{c(\alpha)-\sqrt{c(\alpha)^2+4(-1)^{\ell-1}}}{2} . \nonumber
\end{align}
The sequence $(q_n)_{n=1}^{\infty}$ of denominators satisfies the additional recursive relation
\begin{align} \label{qnrel}
q_{n+\ell}=c(\alpha)q_{n}+(-1)^{\ell-1}q_{n-\ell}, \qquad n\geq 2\ell.
\end{align}
For $k= 0,1, \ldots,\ell-1$ we set
\begin{equation} \label{ekck}
c_k=\frac{q_{\ell+k}-bq_k}{a-b} \qquad \text{ and } \qquad e_k=(-1)^{k-1}\frac{|aq_k-q_{\ell+k}|}{q_{\ell}} \cdot 
\end{equation} 
For notational convenience we extend the definitions of $c_k$ and $e_k$ to all integers $k\geq 0$ periodically modulo $\ell,$ so that in particular we have $c_\ell = c_0$ and $e_\ell = e_0.$  We also make use of the following relations  (for more details see e.g. \cite{GN18}):
\begin{equation} \label{ntherror} \begin{gathered}
c_k >0,   \\ 
 \Lambda_{m \ell +k} := q_{m \ell +k}\alpha - p_{m \ell +k } =  e_kb^m  = (-1)^{m\ell + k + 1}|e_k b^m|,  \\  
\frac{1}{q_{m \ell +k}}  =  \calO(|b|^m), \quad m\to \infty   \\
q_{m \ell + k} |b|^m =  c_k + \calO(b^{2m})  , \quad m \to \infty .  \end{gathered}
\end{equation}

 When studying the irrational $\alpha=[0; \overline{a_1,\ldots,a_{\ell}}]$, it is useful to consider two families of permutations on  $\ell-$tuples of positive integers $\boldsymbol{a}= (a_1,\ldots,a_{\ell})$. For $k=0,1,\ldots,\ell-1$ we define the permutation operator $\tau_k : \mathbb{N}^\ell \rightarrow \mathbb{N}^\ell$ by 
 $$ \tau_k(\boldsymbol{a}) = (a_{k+1},\ldots, a_{\ell}, a_1,\ldots, a_k). $$
 Likewise, we define the permutations $\sigma_k:\mathbb{N}^\ell \rightarrow \mathbb{N}^\ell$ for  $k=2,\ldots, \ell $ by 
 $$ \sigma_k(\boldsymbol{a}) = (a_{k-1},\ldots,a_1,a_{\ell},\ldots, a_k)  $$
 while for $k=1$ we set $\sigma_1(\boldsymbol{a}) = (a_{\ell},\ldots,a_1) $. We can define $\tau_k$ and $\sigma_k$ for all $k\geq 1$ by extending the definitions above periodically modulo $\ell$. Given a purely periodic irrational $\alpha$ with period $\boldsymbol{a}$, the corresponding purely periodic irrationals with periods $\tau_k(\boldsymbol{a})$ and $\sigma_k(\boldsymbol{a})$ will be denoted by
 \begin{equation} \label{permuations}
 \alpha_{\tau_k} = [0;\overline{a_{k+1},\ldots,a_{\ell},a_1,\ldots,a_k}]\quad \text{ and }\quad \alpha_{\sigma_k}=[0;\overline{a_{k-1},\ldots a_1,a_{\ell},\ldots,a_k}] .
 \end{equation}
 
The significance of the permutations $\tau_k$ and $\sigma_k$ when studying the approximation properties of $\alpha$ is indicated by the following relations, which hold for any index $k=0,1,\ldots, \ell-1:$
\begin{equation}\label{qlplusone} \begin{gathered}
c(\alpha) \, =\, c(\alpha_{\tau_k}) \, = \, c(\alpha_{\sigma_k} ) ,   \\ 
q_{\ell}(\alpha_{\tau_k}) \, = \, q_{\ell}(\alpha_{\sigma_k}) ,  \\ 
p_{\ell}(\alpha_{\tau_k}) = q_{\ell-1}(\alpha_{\sigma_k}) \, \text{ and } \, p_{\ell}(\alpha_{\sigma_k}) = q_{\ell-1}(\alpha_{\tau_k}),   \\  
\frac{q_{\ell+1}(\alpha_{\tau_k})}{q_{\ell}(\alpha_{\tau_k})} = a_k + \frac{p_{\ell}(\alpha_{\sigma_k})}{q_{\ell}(\alpha_{\sigma_k})} ,   \\
| c_ke_k| = \frac{q_{\ell}(\alpha_{\tau_k})}{ c(\alpha_{\tau_k})-2b} \, \cdot 
\end{gathered}\end{equation}

\subsection{Main Results} We are now equipped to state our main results. As alluded to above, our first goal is to generalize the convergence result of \cite{ATZ20} on perturbed products to irrationals $\beta= [0;b_1,\ldots,b_h,\overline{a_1,\ldots, a_{\ell}}]$ with $\ell\geq 2$. For any $\varepsilon \in\mathbb{R}$ we define
\begin{equation} \label{perturbedproductnewdef}
P_{q_n}(\beta,\varepsilon):=\prod_{r=1}^{q_n}2\big| \sin \pi \big(r\beta + (-1)^{n+1}\frac{\varepsilon}{q_n}\big)\big|.
\end{equation} 

 In view of the aforementioned theorem  by Grepstad and Ne\"{u}muller in \cite{GN18}, one would expect the perturbed products to converge along specific subsequences. We show that this is indeed the case. For the sake of convenience, we introduce the notation 
\begin{equation} \label{ukt}
u_k(t)=2\left(  \frac{t}{|e_kc_k|}-\{t\alpha_{\sigma_k}\}+\frac12 \right), \quad t=1,2,\ldots
\end{equation}
for each $k=1,\ldots, \ell$,  where $c_k, e_k$ are as in \eqref{ekck} and $\alpha_{\sigma_k}$ as in \eqref{permuations}, all referring to the purely periodic irrational $\alpha = [0; \overline{a_1,\ldots,a_\ell}]$.

\begin{thm}\label{thm:limitFunc} 
Let  $\beta= [0;b_1,\ldots, b_h, \overline{a_1,\ldots,a_{\ell}}]$ where  $h\geq 0$, $\ell,a_1, \ldots,a_{\ell} \in \mathbb{N}$ and $P_{q_n}(\beta,\ve)$ be the sequence of perturbed Sudler products defined in \eqref{perturbedproductnewdef}. Then for each $k=1,\ldots, \ell$ the subsequence $P_{q_{h+m\ell+k}}(\beta,\ve)$ converges locally uniformly to a function $G_{k}(\beta,\ve)$. The limit function satisfies 
\begin{align} \label{geven}
G_{k}(\beta,\varepsilon) =&   \left|1 + \frac{\varepsilon}{|c_ke_k|}\right|  \left(1 +\frac{1}{|b|^{2}}\right)^{\frac{1}{c-2}} \frac{1 }{(c!)^{1/(c-2)}}\times  \nonumber \\ 
& \times \prod_{t=1}^{\infty}\left| \left( 1 - \frac{\left(1 +  \frac{2\varepsilon}{|e_kc_k|}\right)^2}{u_k(t)^2} \right)\left(1-\frac{\left(1 + \frac{2}{|b|^2}\right)^2}{u_k(t)^2} \right)^{\frac{1}{c-2}}\prod_{s=1}^{c-1} \left( 1 - \frac{(1+2s)^2}{u_k(t)^2} \right)^{-\frac{1}{c-2}} \right|
\end{align}
when $\ell$ is even, and
\begin{equation} \label{godd}
G_{k}(\beta,\varepsilon) = \frac{\left| 1 + \dfrac{\varepsilon}{|c_ke_k|}\right|}{\prod\limits_{s=1}^{c}\left|s-a \right|^{\frac{1}{c}}} \times \prod_{t=1}^{\infty}\left|\left(1 - \frac{\left(1+ \frac{2\varepsilon}{|e_kc_k|}  \right)^2}{u_k(t)^2} \right)\prod_{s=0}^{c-1}\left(1 - \frac{\left(1 + 2( s- \frac{1}{|b|})\right)^2 }{u_k(t)^2}  \right)^{-\frac{1}{c}} \right|
\end{equation}
when $\ell$ is odd. Here the sequence $(u_k(t))_{t=1}^{\infty}$ is given in \eqref{ukt}  and the constants $a$, $b$, $c$, $c_k$ and $e_k$ are defined in Section \ref{subsec:contfracprelim}, all corresponding to the purely periodic irrational $\alpha=[0;\overline{a_1,\ldots,a_{\ell}}]$. In both cases, the functions $G_k(\beta,\cdot)$, $k= 1,\ldots,\ell$, are continuous and $C^{\infty}$ on every interval where they are non-zero. 
\end{thm}
 
The formulae \eqref{geven} and \eqref{godd} in Theorem \ref{thm:limitFunc} imply that the limit functions $G_k(\beta,\varepsilon)$ only depend on the periodic part  of the continued fraction expansion of $\beta$; the digits $b_1,\ldots,b_h$ in the pre-periodic part do not play any role at all.  
 
\begin{rem}
Note that we have altered the definition of $P_{q_n}(\alpha,\varepsilon)$ compared to \cite{ATZ20}, i.e.\ we use $(-1)^{n+1}$ instead of $(-1)^n$. This relates to the fact that in \cite{GN18} the denominator of the $n$--th convergent was defined as $q_{n+1}$ while in \cite{ATZ20} the denominator of the $n$--th convergent is $q_{n}$.
\end{rem}

\begin{rem}
An alternative proof of Theorem \ref{thm:limitFunc} has recently appeared in  \cite{chris_bence}. There the limit function is given in a different form and additionally an explicit approximation error is obtained.   
\end{rem}

Since the constants $C_1,\ldots,C_{\ell}$ in \eqref{constants} satisfy $C_k = G_k(\beta,0), (1\leq k \leq \ell)$, Theorem \ref{thm:limitFunc} allows us to explicitly calculate their values.

\begin{cor}\label{cor:C_k} Let $\beta=[0;b_1,\ldots,b_h, \overline{a_1,\ldots,a_{\ell}}]$ and $C_1,\ldots, C_{\ell}>0$ be the constants in \eqref{constants}. Then for $k=1,2,\ldots, \ell $ we have
\begin{equation} \label{cvalueseven}
 C_k = \left( \frac{1+a^2}{c!} \right)^{\frac{1}{c-2}}\prod_{t=1}^{\infty} \left(1- \frac{1}{u_k(t)^2}\right) \left|1- \frac{(1+2a^2)^2}{u_k(t)^2} \right|^{\frac{1}{c-2}} \prod_{s=1}^{c-1}\left| 1- \frac{(1+2s)^2}{u_k(t)^2}  \right|^{-\frac{1}{c-2}}
\end{equation}	
when $\ell $ is even, and 
\begin{equation} \label{cvaluesodd}
C_k = \frac{1}{\prod\limits_{s=1}^{c}\left|s-a\right|^{\frac{1}{c}}}\prod_{t=1}^{\infty}\left(1- \frac{1}{u_k(t)^2}\right)\prod_{s=0}^{c-1}\left| 1- \frac{(1+2s-2a)^2 }{u_k(t)^2} \right|^{-\frac{1}{c}}
\end{equation}	
when $\ell $ is odd. 	
\end{cor}

Our next result  relates the asymptotic size of $P_N(\beta)$ with the size of the constants $C_1,\ldots, C_{\ell}$ in \eqref{constants}. This is the analogue of Lemma $1$ in \cite{ATZ20}, and the proof is nearly identical. Nevertheless, we include the proof later in the text for the sake of completeness. 

\begin{thm} \label{thm2}
Let $\beta = [0;b_1,\ldots, b_h,\overline{a_1, \ldots, a_\ell} ]$ and $(C_k)_{k=1}^{\ell}$ be the constants as in \eqref{constants}. If $C_{k_0}< 1$ for some index $1\leq k_0 \leq \ell$ then 
\begin{equation} \label{sudler_behaviour} \liminf_{N\to\infty}P_N(\beta) = 0 \qquad \text{ and } \qquad \limsup_{N\to\infty}\frac{P_N(\beta)}{N} =\infty . \end{equation}
\end{thm}
 \begin{rem} \label{rem4} By the aforementioned Theorem of Grepstad and Neum\"uller, the values of $C_1,\ldots, C_k$ only depend on the periodic part of the quadratic irrational  $\beta$. Combined with Theorem \ref{thm2}, this explains why it suffices to consider only purely periodic irrationals when trying to detect those irrationals $\beta$ for which the Sudler product $P_{N}(\beta)$ satisfies \eqref{sudler_behaviour}. 

\end{rem}

Theorem \ref{thm2} tells us that as long as \emph{one} of the constants $C_k\,\, (1\leq k \leq \ell)$ defined in \eqref{constants} is less than $1$, the Sudler product corresponding to the irrational $\beta = [0;b_1,\ldots, b_h,\overline{a_1, \ldots, a_\ell} ]$ satisfies  \eqref{sudler_behaviour}.  This raises the question of which out of the $\ell$ constants $C_k$ associated with  $\alpha=[0;\overline{a_1, \ldots a_{\ell}}]$  is expected to be minimal.

The plots in Figures \ref{fig:alph1a_2} and \ref{fig:alph2a_2} show graphs of the functions $G_k(\alpha, \varepsilon)$ for specific choices of $\alpha=[0; \overline{a_1, a_2}]$ and $k \in \{1, 2\}$. Since $C_k =G_k(\alpha, 0),$ the value of $C_k$ is the ordinate of the point of intersection of the graph with the vertical axis. These graphs seem  to suggest that the bigger the digit $a_k$ is, the smaller the constant $C_k$ becomes.  

\begin{figure}
    \centering
    \subcaptionbox{$\alpha=[0;\overline{1,2}]$}{\includegraphics[width=7cm]{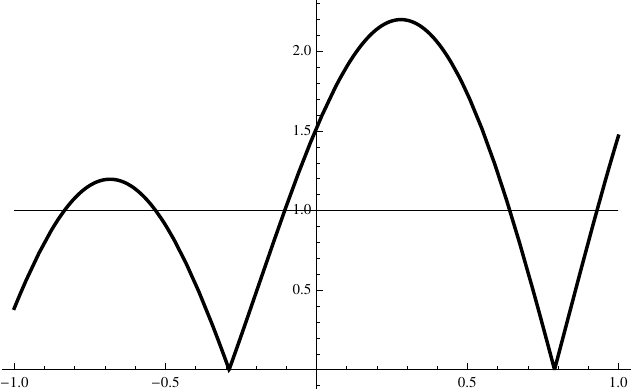}}
    \subcaptionbox{$\alpha=[0;\overline{1,3}]$}{\includegraphics[width=7cm]{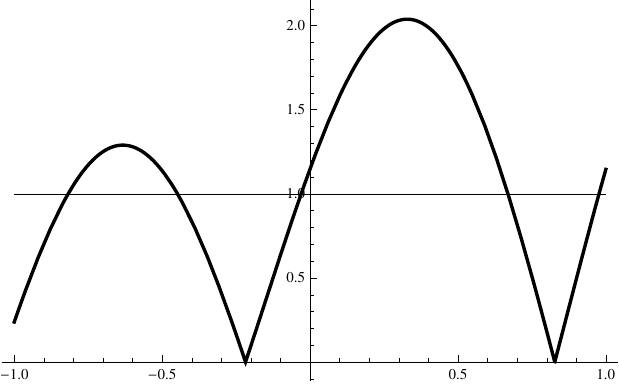}}
    \subcaptionbox{$\alpha=[0;\overline{1,4}]$}{\includegraphics[width=7cm]{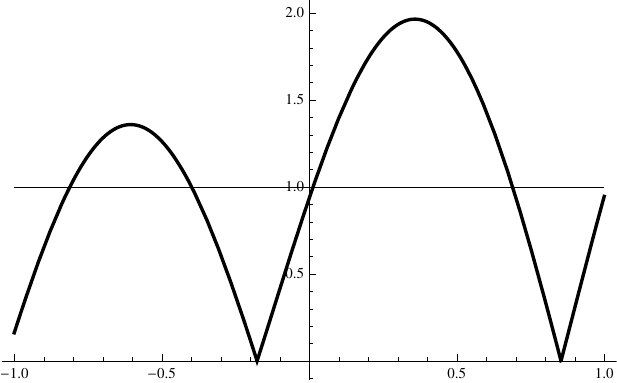}}
    \subcaptionbox{$\alpha=[0;\overline{1,5}]$}{\includegraphics[width=7cm]{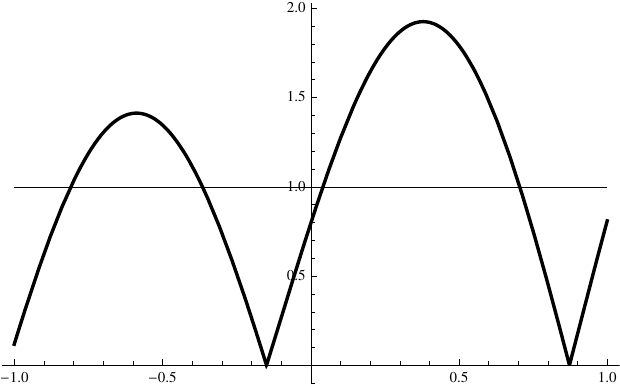}}
    \caption{Plots of the limit functions $G_2(\alpha, \varepsilon)$ for the stated values of $\alpha=[0; \overline{1, a_2}]$. It appears that $C_2=G_2(\alpha,0)<1$ whenever $a_2\geq 4$. \label{fig:alph1a_2}}
\end{figure}

\begin{figure}
    \centering
    \subcaptionbox{$\alpha=[0;\overline{2,3}]$}{\includegraphics[width=7cm]{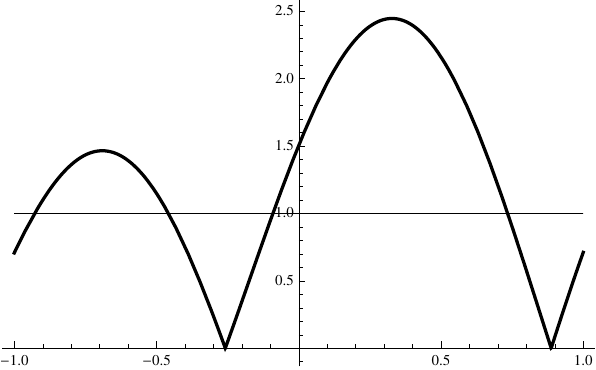}}
    \subcaptionbox{$\alpha=[0;\overline{2,4}]$}{\includegraphics[width=7cm]{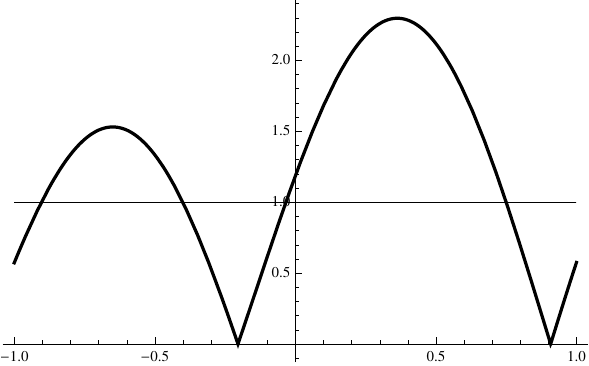}}
    \subcaptionbox{$\alpha=[0;\overline{2,5}]$}{\includegraphics[width=7cm]{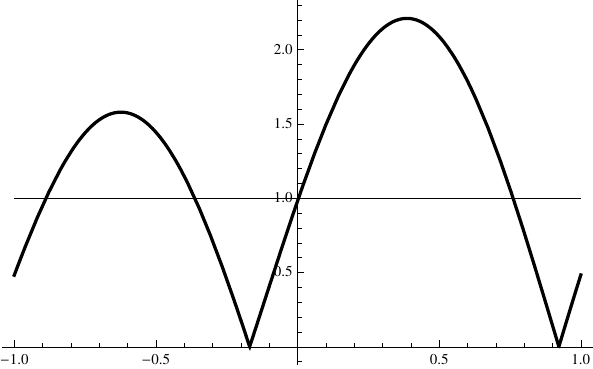}}
    \subcaptionbox{$\alpha=[0;\overline{2,6}]$}{\includegraphics[width=7cm]{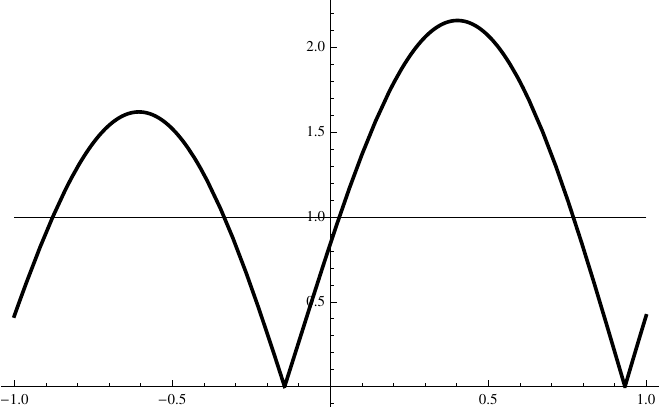}}
    \caption{Plots of the limit functions $G_2(\alpha, \varepsilon)$ for the stated values of $\alpha=[0; \overline{2, a_2}]$. It appears that $C_2=G_2(\alpha,0)<1$ whenever $a_2\geq 5$. \label{fig:alph2a_2}}
\end{figure}

In spite of the hints provided by the plots, it remains to verify rigorously that $C_k=G_k(\alpha,0)$ decreases with increasing values for the digit $a_k$. Moreover, it should be pointed out that the given plots provide no information on the significance of the period length $\ell$ in the continued fraction expansion of $\alpha$. 
The period length can be chosen arbitrarily large, and it might be that the size of $\ell$ has an impact on the overall sizes of the constants $C_k$. Moving forward, we will thus focus on three main questions:
\begin{enumerate}
\item[$\bullet$] Is the phenomenon implied by the graphs in Figures \ref{fig:alph1a_2} and \ref{fig:alph2a_2} indeed true, i.e.\ for any index $1\leq k \leq \ell$, is  $C_k$ decreasing as a function of the digit $a_k$?   \vspace{2mm}
\item[$\bullet$] Suppose we fix some period length $\ell \geq 2$. Does there exist an integer $K = K_\ell\geq 1$ such that for any irrational   $\alpha = [0; \overline{a_1, \ldots, a_\ell} ]$ with $\max_{1\leq i \leq \ell} a_i \geq K$ the Sudler product $P_N(\alpha)$ satisfies \eqref{sudler_behaviour}? \vspace{2mm}
\item[$\bullet$] If such an integer exists, can it be chosen independently of the period length $\ell$? \end{enumerate}

By a careful analysis of the product formulas established in Corollary \ref{cor:C_k} we find that when $k$ is the index  corresponding to the maximal digit $a_k$ in $\alpha=[0; \overline{a_1, \ldots a_{\ell}}],$ then $C_k$ is bounded above by an expression which is indeed decreasing as a function of $a_k$. In turn, this leads to the following result, which provides a positive answer to questions two and three.
\begin{thm} \label{thm:ak23}
Let $\beta=[0;  b_1,\ldots,b_h,  \overline{a_1, \ldots , a_{\ell}}]$ be a quadratic irrational with period length $\ell \geq 2$, and say $a_k=\max_j a_j$. Then 
$$ \liminf_{N\to\infty}P_N(\beta) = 0 \qquad \text{ and } \qquad \limsup_{N\to\infty}\frac{P_N(\beta)}{N} =\infty , $$
whenever $a_k \geq 23$.
\end{thm}
\begin{rem}
Recall that it was shown by Lubinsky that $\liminf P_N(\alpha)=0$ whenever $\alpha$ has unbounded coefficients in its continued fraction expansion \cite{lubinsky}. In fact, Lubinsky made the more striking observation that there exists a cutoff value $a_k\geq K$ for which Theorem \ref{thm:ak23} is true, not only for quadratic irrationals but for any irrational $\alpha$.   Note, however, that Lubinsky's approach merely tells us that $K \approx e^{800}$ will suffice. 
Theorem \ref{thm:ak23} is thus a significant improvement of the best known cutoff value $K$ for quadratic irrationals.
\end{rem}
Theorem \ref{thm:ak23} can be seen as a partial analogue of the second part of the aforementioned theorem by Aistleitner, Technau and Zafeiropoulos. We will not attempt to imitate the first part of their result, stating that $\liminf P_N(\alpha) >0$ for sufficiently small values of $\max_i a_i$.   It will be clear from the proof of Theorem 3 that the role played by the period length $\ell$ is not fully understood, and in light of this we leave the following open problems for further discussion. 
\begin{ques*}
Let $\alpha=[0;\overline{a_1, a_2, \ldots, a_{\ell}}]$ be a quadratic irrational with $a_k = \max_i a_i$.
\begin{itemize}[itemsep=5pt]
\item According to Theorem \ref{thm:ak23}, there exists an integer $K=K_\ell\geq 1$ such that \eqref{sudler_behaviour} holds whenever $a_k \geq K_\ell$, and $K_{\ell}\leq 23$ for all period lengths $\ell$. However, if we fix $\ell$, what is then the optimal value of $K_\ell$? We will see in the proof (see Section \ref{sec:proofthm3}) that for odd periods $\ell$, Theorem \ref{thm:ak23} holds for $K= 22$. Moreover, for the special case when $\alpha =[0;\overline{1, a_2}]$, Theorem \ref{thm:ak23} holds for $K=21$, and the plots in Figure \ref{fig:alph1a_2} and \ref{fig:alph2a_2} suggest that we can actually do \emph{much} better. This brings us to the following question.
\item Is it possibly true that  
if $a_k \geq 6$, then $\liminf \limits_{N \to \infty} P_N(\alpha) = 0$ and $\limsup \limits_{N\to\infty}\dfrac{P_N(\alpha)}{N} = \infty$? In other words, is $K_\ell \leq 6$ for all $\ell\geq 1$?  Numerical evidence seems to suggest that the answer is positive, and that the threshold value $K=6$ established for irrationals $\beta = [0;b,b,\ldots]$ in \cite{ATZ20} might in fact be a universal bound for all quadratic irrationals. 
\end{itemize} 
\end{ques*}

Finally, we point out that Aistleitner and Borda have shown the following duality in \cite{chris_bence}: for any badly approximable $\alpha,$ we have 
\[ \liminf_{N\to \infty} P_N(\alpha) = 0  \quad \text{ if and only if } \quad \limsup_{N\to \infty} \frac{P_N(\alpha)}{N} = \infty. \]
Thus for a fixed period length $\ell \geq 2,$ giving a complete characterisation of the quadratic irrationals $\alpha$
for which \eqref{sudler_behaviour} holds also determines those irrationals $\alpha$ for which $\liminf\limits_{N\to \infty} P_N(\alpha) > 0$ and $\limsup\limits_{N\to\infty}P_N(\alpha)/N < \infty.$

\subsection{Oragnization of the paper}
The remainder of the paper is organized as follows. Theorems \ref{thm:limitFunc} and \ref{thm2} are proved in Sections \ref{sec:proofthm1} and \ref{sec:proofthm2}, respectively. In Section \ref{sec:Gx}, we analyse the product 
\begin{equation*}
    G(x) = \prod_{t=1}^{\infty} \left( 1- \frac{x^2}{u_k(t)^2}\right),
\end{equation*}
with $u_k(t)$ as defined in \eqref{ukt}. This product plays a crucial role in the expressions for $C_k$ in Corollary \ref{cor:C_k}. Note that the sequence $\cu = (u_k(t))_{t \in \mathbb{N}}$ can be viewed as a perturbation of the arithmetic progression $(2t/|e_kc_k|)_{t \in \mathbb{N}}$,  so it is natural to compare $G(x)$ to the product
\begin{equation*}
    \prod_{t=1}^{\infty} \left( 1-\frac{x^2}{\left( 2t/|e_kc_k|\right)^2}\right).
\end{equation*}
In doing so, we obtain Theorem \ref{thm:Gbounds}, which tells us that 
\begin{equation}\label{eq:Gboundpreversion}
K_1 \frac{\text{dist} (x,\cu)}{|x|} \leq |G(x)| \leq K_2 \frac{1}{|x|},
\end{equation}
for appropriate constants $K_1$ and $K_2$.

In Section \ref{sec:sizeCk} we use Corollary \ref{cor:C_k} and Theorem \ref{thm:Gbounds} to find an upper bound on $C_k$ in \eqref{constants} for both odd and even periods $\ell$. It turns out that $C_k$ can be bounded by expressions which clearly decrease to zero as $a_k \to \infty$ (see Theorems \ref{thm:boundCk} and \ref{thm:boundCkodd}). The speed of decay depends on the constants $K_1$ and $K_2$ in \eqref{eq:Gboundpreversion}, and for this reason we treat separately the case $\alpha=[0;\overline{a_1, a_2}]$ where either $a_1=1$ or $a_2=1$ (as better constants $K_1$ and $K_2$ can then be found).

Finally, in Section \ref{sec:proofthm3} we use the bounds obtained for $C_k$ in Section \ref{sec:sizeCk} to show that $C_k<1$ whenever $\alpha=[0;\overline{a_1, a_2, \ldots , a_k}]$ with $\max_j a_j \geq 23$. By Theorem \ref{thm2}, this proves Theorem \ref{thm:ak23}.

\section{Proof of Theorem \ref{thm:limitFunc}\label{sec:proofthm1}}
Theorem \ref{thm:limitFunc} states that when $\beta=[0;b_1,\ldots,b_h,\overline{a_1, \ldots , a_{\ell}}]$, the perturbed Sudler products $P_{q_n}(\beta, \varepsilon)$ in \eqref{perturbedproductnewdef} converge along subsequences $(q_{h + m \ell +k})_{m=1}^{\infty}$ to explicit limit functions 
\[G_k(\beta, \varepsilon) = \lim_{m \to \infty}  P_{q_{h+m \ell +k}}(\beta, \varepsilon), \quad k=1, \ldots , \ell. \]
 We first present the proof of Theorem \ref{thm:limitFunc} for the purely periodic quadratic irrational $\alpha=[0;\overline{a_1,\ldots,a_\ell}]$ (i.e. when $h=0$). We then briefly explain how the proof can be generalised for arbitrary quadratic irrationals $\beta=[0;b_1,\ldots,b_h,\overline{a_1, \ldots , a_{\ell}}].$ 
 
 Our first observation is that $P_{q_n}(\alpha, \varepsilon)$ can be decomposed into a product of three factors. 

\begin{lem}\label{lem:decompPqnPerturbed}
	Let $\alpha=  [0;\overline{a_1,\ldots, a_{\ell} }]$ and let $q_n$ denote the denominator of the $n$-th convergent of $\alpha$. Then for any $\ve\in \mathbb{R}$,  
	$$P_{q_n}(\alpha,\varepsilon)=A_n(\alpha, \varepsilon) \cdot B_n(\alpha)  \cdot  C_n(\alpha,\varepsilon),$$
	where 
	\begin{align*}
		A_n(\alpha,\varepsilon) &= 2q_n\Big|\sin\pi\Big( \Lambda_n + (-1)^{n+1}\frac{\varepsilon}{q_n}\Big)\Big|\\
		B_n(\alpha) &= \left| \prod_{t=1}^{q_{n}-1} \frac{s_{n}(t)}{2 \sin (\pi t / q_{n})} \right|,
		\label{eq:Bm} \\
		C_n(\alpha,\varepsilon) &= \prod_{t=1}^{q_n-1}\left(1-\frac{s^2_n(0,\varepsilon)}{s^2_n(t)}\right)^{\frac12} , \\
	\end{align*}
 and
	\begin{align*}	
		s_{n}(0,\varepsilon)=2 \sin \pi\left(\frac{\Lambda_n}{2} + (-1)^{n+1}\frac{\varepsilon}{q_n}\right), \quad s_{n}(t) =2 \sin \pi\left(\frac{t}{q_{n}} - |\Lambda_{n}| \left(\left\{ \frac{
			tq_{n-1}} {q_{n}}\right\}-\frac{1}{2}\right)\right).
	\end{align*}
\end{lem}

Lemma \ref{lem:decompPqnPerturbed} is the natural analogue of Lemma $5.1$ of \cite{mv} for the product $P_{q_n}(\phi)$ and Lemma $4.2$ in \cite{GN18} for $P_{q_n}(\alpha)$. We omit the proof since it is nearly identical, the only difference being that it involves an additional term within the argument of the sine. We continue by analysing the behaviour of each of the three factors $A_n(\alpha,\ve), B_n(\alpha) \text{ and } C_n(\alpha,\ve)$. 

The factor $B_n=B_n(\alpha)$ is independent of the perturbation argument $\ve$, and it is shown in \cite{GN18} that for each $k=1,2,\ldots, \ell$ the limit 
$$B^{(k)} \, = \, \lim_{m\to\infty}B_{m \ell +k}$$ 
exists. Regarding the factor $A_{m \ell +k}(\alpha,\ve)$, we have 
$$ A_{m \ell +k}(\alpha,\varepsilon) = 2\pi \big||c_ke_k|+\varepsilon\big| + \calO(b^{2m}), \quad m\to\infty,   $$
therefore  $A_{m \ell +k}(\alpha,\ve)$ converges to $  2\pi \big||c_ke_k|+\varepsilon\big| .$
Finally we need to establish convergence for the factor $C_n(\alpha,\varepsilon)$. Here we can argue as in \cite[Section $6$]{mv}, taking into account that the factor $s_n(0,\ve)$ depends on the parameter $\ve$ and satisfies
$$ |s_{m \ell +k}(0,\ve)|  \sim \pi\left| |e_kb^m |+ 2\ve \right|  + \calO(b^{2m}), \quad m\to\infty,  $$
and also 
\begin{align*}
s_{m \ell +k}(t)=\pi |e_kb^m|u_{k}(t) + \calO(|b|^{m/5}), \quad m\to\infty.
\end{align*}
The same arguments as in \cite{mv} imply that for any $\varepsilon \in \mathbb{R}$, \[ \lim_{m \to \infty}C_{m\ell +k}(\alpha,\varepsilon) = \prod_{t=1}^{\infty}\left|1-\frac{\left(1+\frac{2\varepsilon}{|e_kc_k|}\right)^2}{u_{k}(t)^2}\right| \] 
 
In view of Lemma \ref{lem:decompPqnPerturbed} we deduce that for $k=1,2,\ldots, \ell $ the limiting function $G_k$ satisfies 
\begin{equation} \label{gktemporary} 
G_k(\alpha,\ve) = 2\pi\big||e_kc_k| + \varepsilon \big| B^{(k)} \prod_{t=1}^{\infty}\left|1-\frac{\left(1+\frac{2\varepsilon}{|e_kc_k|}\right)^2}{u_{k}(t)^2}\right|, 
\end{equation}
where $u_k(t)$ is defined in \eqref{ukt}. Arguing as in the proof of Theorem $1$ in \cite{ATZ20} we can show that the convergence is locally uniform. Now we fix a value of $k=1,\ldots, \ell$ and consider indices $n=m\ell + k, m=1,2,\ldots$ In order to determine the formula of $G_k$ we will use relation \eqref{qnrel}. We  distinguish two cases depending on the parity of the period length $\ell$. 

$\bullet$ If $\ell \equiv 0 \pmod 2$, then \eqref{qnrel} gives $c(\alpha)q_n = q_{n+\ell} + q_{n-\ell}$, and thus
\begin{equation}\label{evenequation}
P_{cq_n}(\alpha) = P_{q_{n-\ell}+ q_{n+\ell}} (\alpha) \, .
\end{equation}
The left hand side in \eqref{evenequation} is
\begin{eqnarray*}
P_{cq_n}(\alpha) &=& \prod_{r=1}^{cq_n}2|\sin \pi r\alpha| = \prod_{s=0}^{c-1}\prod_{r=1+sq_n}^{(s+1)q_n}2|\sin \pi r\alpha| = \prod_{s=0}^{c-1}\prod_{r=1}^{q_n}2|\sin \pi(r\alpha + sq_n \alpha) | \\
&\stackrel{\eqref{ntherror}}{=} & \prod_{s=0}^{c-1}\prod_{r=1}^{q_n}2|\sin\pi(r\alpha + (-1)^{n+1}s|e_kb^m| )| = \prod_{s=0}^{c-1}P_{q_n}(\alpha, sq_n|e_kb^m|) \, ,
\end{eqnarray*}
while the right hand side is 
\begin{eqnarray*}
 P_{q_{n-\ell}+ q_{n+\ell}} (\alpha)&=&\!\!\! \prod_{r=1}^{q_{n+\ell}+ q_{n-\ell}} 2|\sin \pi r\alpha | = \prod_{r=1}^{q_{n-\ell}} 2|\sin \pi r\alpha | \cdot \prod_{r=1+q_{n-\ell}}^{q_{n+\ell}+q_{n-\ell}}2|\sin \pi r\alpha | \\
&=& P_{q_{n-\ell}}(\alpha)\cdot \prod_{r=1}^{q_{n+\ell}}2|\sin \pi(r\alpha+q_{n-\ell  }\alpha)| \\
&\stackrel{\eqref{ntherror}}{=}& P_{q_{n-\ell}}(\alpha)\cdot \prod_{r=1}^{q_{n+\ell}}2|\sin \pi(r\alpha+ (-1)^{n+\ell +1}|e_kb^{m-1}|)| \\
&=& P_{q_{n-\ell}}(\alpha) P_{q_{n+\ell}}(\alpha, q_{n+\ell}|e_kb^{m-1}| ) .
\end{eqnarray*}

Since the functions $P_{q_n}(\alpha,\ve)$ converge locally uniformly, letting $m\to \infty$ in \eqref{evenequation} and taking \eqref{ntherror} into account we obtain 
\begin{equation} \label{greleven}
\prod_{s=0}^{c-1}G_k(\alpha, s|c_ke_k|)  = G_k(\alpha,0) \,  G_k\!\left(\alpha, \frac{|c_ke_k|}{|b|^2}\right) .
\end{equation}
Substituting $G_k$ from \eqref{gktemporary} in \eqref{greleven} we obtain \eqref{geven}. 

$\bullet$ If $\ell \equiv 1 \pmod 2$, then \eqref{qnrel} becomes $q_{n+\ell}= c(\alpha)q_n + q_{n-\ell}$, so 
\begin{eqnarray*} 
\frac{P_{q_{n+\ell}}(\alpha)}{P_{q_{n-\ell}}(\alpha)} &=& \prod_{r=1}^{q_{n+\ell}-q_{n-\ell}}2|\sin\pi( r \alpha +q_{n-\ell}\alpha)|=\prod_{r=1}^{c(\alpha)q_{n}}2|\sin\pi (r \alpha +q_{n-\ell} \alpha)|\nonumber  \\
&=& \prod_{s=0}^{c(\alpha)-1}   \prod_{r=1}^{q_n}2|\sin \pi(r\alpha + sq_n\alpha + q_{n-\ell}\alpha )|\\
& \stackrel{\eqref{ntherror}}{=} & \prod_{s=0}^{c(\alpha)-1} \prod_{r=1}^{q_n}2| \sin \pi(r\alpha + (-1)^{n+1}s|e_kb^m|-(-1)^{n+1}|e_kb^{m-1}| )| \\
&=& \prod_{s=0}^{c(\alpha)-1} P_{q_n}\left(\alpha, q_n|e_kb^m|\left(s-\frac{1}{|b|} \right) \right) .
\end{eqnarray*}

By the main result of \cite{GN18} we know that the sequence $P_{q_n}(\alpha)$ converges to a limit $C_k>0$, hence letting $m\to\infty$ in the above equality we get 
\begin{equation*} 
\prod_{s=0}^{c(\alpha)-1}G_k\left(\alpha, |c_ke_k|( s- \tfrac{1}{|b|})\right) =1.
\end{equation*}
Substituting this into \eqref{gktemporary} we obtain \eqref{godd}. This completes the proof of Theorem \ref{thm:limitFunc} when $\alpha=[0;\overline{a_1,\ldots,a_k}]$ is a purely periodic quadratic irrational.

  We now deal with quadratic irrationals for which the length of the pre-period is $h\geq 1.$ If we consider the irrational $\beta=[0;b_1,\ldots,b_h,\overline{a_1,\ldots,a_\ell}]$ and $\alpha=[0;\overline{a_1,\ldots,a_\ell}]$ is as before, we can still find a factorisation
\[ P_{q_n}(\beta, \varepsilon) = A_n(\beta,\varepsilon) \cdot B_n(\beta) \cdot C_n(\beta,\varepsilon),  \]
where the three factors are defined similarly to the purely periodic case; the only difference which appears, is that the parameters $e_k$ and $c_k$ are replaced by   different parameters $c_{h,k}$ and $e_{h,k}$. The definition of these parameters is given in the proof of Corollary 1.3 in \cite{GN18}. \par 
The limits of the three factors of $P_{q_n}(\beta,\varepsilon)$ are: 
\begin{align*}
    \lim_{m\to \infty} A_{h+m\ell + k}(\beta,\varepsilon) & = 2\pi \big| |c_{h,k}e_{h,k}| + \varepsilon \big|, \\
    \lim_{m\to\infty} B_{h+m\ell + k }(\beta)\quad &  = B^{(h,k)}, \quad \text{ and }\\
    \lim_{m\to\infty} C_{h+ m\ell + k}(\beta,\varepsilon) & =  \prod_{t=1}^{\infty} \left|1-\frac{\left(1+\frac{2\varepsilon}{|e_{h,k}c_{h,k}|}\right)^2}{u_{k}(t)^2}\right|.
\end{align*}
We therefore deduce that for $k=1,2,\ldots,\ell$, the subsequence $P_{q_{h+\ell m + k}}(\beta,\varepsilon)$ converges to some limit function $G_k(\beta,\varepsilon).$  We now invoke the fact that 
\[ |c_{h,k}e_{h,k}| = |c_ke_k|, \qquad k=1,\ldots, \ell, \]
(established in the proof of Corollary 1.3 in \cite{GN18}) to deduce that for all $k=1,\ldots, \ell,$  
\[\lim_{m \to \infty} P_{q_{h+m\ell + k}}(\beta, \varepsilon) = G_{k}(\beta, \varepsilon) = \lambda_k G_{k}(\alpha, \varepsilon),\]
for some constant $\lambda_k$. Finally, since we know from \cite[Corollary 1.3]{GN18} that $G_{k}(\beta, 0)= G_k(\alpha, 0)$  for each $k=1, \ldots , \ell$, it follows that $\lambda_k=1$ for every $k$. This shows that adding a pre-period to the continued fraction expansion of $\alpha$ leaves the  limit functions $G_k$ unchanged, completing the proof of Theorem \ref{thm:limitFunc}.

\section{Proof of Theorem  \ref{thm2}\label{sec:proofthm2}}

We now prove Theorem \ref{thm2}, namely show that if for the irrational
$\beta=[0;b_1,\ldots,b_h,\overline{a_1,\ldots,a_\ell}]$, we have $G_k(\beta, 0)=C_k<1$ for some $1 \leq k \leq \ell$, then 
\[\liminf_{N\to \infty} P_N(\beta)=0 \quad \text{ and } \quad \limsup_{N \to \infty} \frac{P_N(\beta)}{N} = \infty .\]
Let $k_0$ denote the index for which $C_{k_0}<1$, and fix some $\lambda$ such that $C_{k_0} < \lambda <1$. Since $G_{k_0}$ is continuous at $0$ and $G_{k_0}(\beta,0)=C_{k_0}<\lambda$, there exists $\eta>0$ such that $G_{k_0}(\beta, \ve)<\lambda $ for all $|\ve|<\eta.$ Consider a subsequence $(m_i)_{i=1}^{\infty}$ of $(q_{h+m\ell+k_0})_{m=1}^{\infty}$ such that \begin{enumerate}
	\item[(i)] $m_{i+1} \geq 2m_i,\quad  i=1,2,\ldots \quad$ and 
	\item[(ii)]  $ \|m_{i+1} \beta\| < \dfrac{\eta}{4m_i} ,\quad i=1,2,\ldots$ \vspace{-3mm}
\end{enumerate}
where $\| x\| = \min\{|x-k|: k\in \mathbb{Z}\}$ denotes the distance of $x\in\mathbb{R}$ to the nearest integer. We set
$$ N_i =  m_i + \ldots + m_1  \quad \text{and} \quad
M_j = N_i - N_j, \quad   i\geq 1,\quad   j=1,\ldots, i .   $$
Then
\begin{eqnarray*}
P_{N_i}(\beta)\, & = \,& \prod_{r=1}^{N_i} 2|\sin \pi r \beta| \,\, = \,\, \prod_{j=1}^{i}\prod_{r=1}^{m_j}2|\sin \pi (r+M_j) \beta| \nonumber\\
& = \, & \prod_{j=1}^{i} \prod_{r=1}^{m_j}2 \left|\sin \pi\left(r\beta + \frac{\ve_j}{m_j}\right)\right| \,\, = \,\, \prod_{j=1}^{i}P_{m_j}(\beta,(-1)^{\delta_j}\ve_j) , \label{prod_nk}
\end{eqnarray*}
where $\ve_j= \pm m_j\|M_j \beta\|$ and $\delta_j = 0 \text{ or } 1.$   We see that
\begin{eqnarray*} 
 |\ve_j|\,\,  \leq \,\,   m_j \left( \|m_{j+1} \beta\| + \ldots + \|m_k \beta\| \right) \, < \, \frac{\eta}{2} \cdot
\end{eqnarray*}
By the choice of $\eta$, this implies that $P_{m_j}(\beta,(-1)^{\delta_j}\ve_j)< \lambda$ for all $j$ large enough, hence $ P_{N_i}(\beta) \ll \lambda^i, \, i\to\infty$. This shows that $\liminf\limits_{N\to\infty}P_N(\beta)=0$.

Now set also $T_i = m_{i+1}- (N_i+1), \, i\geq 1$, so that
\begin{equation} \label{ptk}
P_{T_i}(\beta) = \frac{P_{m_{i+1}-1}(\beta)}{\prod\limits_{r=N_i+1}^{m_{i+1}-1}\!\!2|\sin \pi r \beta|} = \frac{P_{m_{i+1}-1}(\beta)}{\prod\limits_{j=1}^{i}\prod\limits_{r=1}^{m_j} 2|\sin \pi (r+ M_j-m_{i+1}) \beta|} \cdot
\end{equation}

At this point we need to point out a simple fact which follows from the proof of Theorem $1.1$ in \cite{GN18} but is not explicitly stated in the text. If $(C_k)_{k=1}^{\ell }$ are the constants in \eqref{constants}, then for each $k=1,\ldots, \ell $ we have 
$$ \lim_{m\to\infty} \frac{P_{q_{h+m \ell + k}-1}(\beta)}{q_{h+ m \ell + k}} = \frac{C_k}{2\pi|c_ke_k|} \, ,  $$
where $(c_k)_{k=1}^{\ell }$ and $(e_k)_{k=1}^{\ell }$ are defined in \eqref{ekck}. Armed with this observation we deduce that the numerator in \eqref{ptk} is $P_{m_{i+1}-1}(\beta)  \asymp m_{i+1} \asymp T_i,\, i\to\infty$, while for the denominator in \eqref{ptk} we can show arguing as in the previous step that 
$$ \liminf_{i\to\infty}  \prod\limits_{j=1}^{i}\prod\limits_{r=1}^{m_j} 2|\sin \pi (r+ M_j-m_{i+1}) \beta| = 0.  $$ 
Therefore $$\limsup_{N\to\infty}\frac{P_N(\beta)}{N} =  \limsup_{i\to\infty}\frac{P_{T_i}(\beta)}{T_i} = \infty .$$

\section{A perturbed sinc product \label{sec:Gx}}
In Sections \ref{sec:sizeCk} and \ref{sec:proofthm3}, our aim will be to determine when the limit $C_k$ in \eqref{constants} is guaranteed to be less than one. Corollary \ref{cor:C_k} suggests that we will need to differentiate between two cases, depending on the parity of the period length $\ell$. Common to both cases is the need for appropriate upper and lower bounds on the function 
\begin{equation} \label{gdef}
G(x) = \prod_{t=1}^{\infty}\left(1 - \frac{x^2}{u_k(t)^2} \right) , \qquad x\in \mathbb{R},  
\end{equation}
where we recall from \eqref{ukt} that 
\begin{equation*}
u_k(t) = 2 \left( \frac{t}{|e_kc_k|}-\{t\alpha_{\sigma_k}\} + \frac{1}{2} \right).  
\end{equation*}

Let us now set
\begin{equation}\label{eq:Aanddelta}
      A=2|e_kc_k|^{-1} \quad \text{ and } \quad \delta_t=1-2\{t\alpha_{\sigma_k}\},
\end{equation}
so that $u_k(t) = At+\delta_t$.
The function $G(x)$ in \eqref{gdef} can then be seen as a perturbed version of the well-known product 
\begin{equation*}
\frac{\sin (\pi A^{-1}x)}{\pi A^{-1}x} = \prod_{t=1}^{\infty} \left(1-\frac{x^2}{(At)^2} \right).
\end{equation*}
It is not difficult to show that if the perturbations $\delta_t$ satisfy $|\delta_t|\leq \delta <A/4$ $(t\in \mathbb{N})$ for some $\delta>0$, then the function $G$ obeys the bounds \begin{equation*}
    C_1 \frac{\mathrm{dist}\left(x, \cu \right)}{|x|^{1+4\delta/A}} \leq |G(x)| \leq C_2 \frac{1}{|x|^{1-4\delta/A}},
\end{equation*}
for constants $C_1$ and $C_2$, where $\mathrm{dist}\left(x, \cu \right)=\min\{|x-At-\delta_t| \, : \, t\in \mathbb{N}\}$. This is related to Kadec's $1/4$-rule \cite{kadec}, and can e.g.\ be seen as a consequence of \cite[Lemma 4]{avdonin}. Due to the low-discrepancy property of Kronecker sequences, the sequence $(\delta_t)_{t}$ satisfies a much stronger condition, which in turn enables us to establish stronger bounds on $G$.

\begin{thm}\label{thm:Gbounds}
Let $\alpha=[0;\overline{a_1, \ldots ,a_{\ell}}]$, and let $k \leq \ell$ be an index satisfying $a_k = \max_j a_j$. Then for $|x|\geq A/2 = |e_kc_k|^{-1}$, the function $G$ in \eqref{gdef} satisfies
\begin{equation} \label{Grate}
	\frac{2}{\pi}e^{-f(a_k)}\left(1 - \frac{2}{3Am}\right)\left(1- \frac{1}{Am}\right)^2 \, \frac{ \mathrm{dist}(x,\cu)}{|x|}\, \,\leq\, |G(x) |\,\, \leq\,\, \frac{14A}{9}e^{f(a_k)} \frac{1}{|x|}, 
\end{equation}
where the positive integer $m=m(x)\geq 1$ is such that $|Am-x| = \min\{|An -x| : n \in \mathbb{N} \}$ and 
\begin{equation}\label{eq:fa_k}
f(a_k) = \frac{13.7}{a_k}+\frac{1}{20 \log a_k}+\frac{1}{100}+\frac{2}{a_k^2}.
\end{equation}
For $|x|<A/2$, we have the bound 
\begin{equation}\label{eq:Gsmallx}
	\frac{2}{\pi}e^{-f(a_k)} \leq |G(x)| \leq 1.
\end{equation}
\end{thm}
\begin{rem}
Notice that since $f(a_k) \to 0.01$ as $a_k \to \infty$, equation \eqref{Grate} reads
\begin{equation*}
    K_1 \frac{ \mathrm{dist}(x,\cu)}{|x|} \leq |G(x)| \leq K_2 \frac{1}{|x|},
\end{equation*}
with $K_1 \approx 2/\pi$ and $K_2 \approx 14A/9$ whenever $a_k$ and $x$ are large. Similar bounds can be established when $k$ is \emph{not} the index of a maximal continued fraction coefficient of $\alpha$, but the size of the constants $K_1$ and $K_2$ will then depend both on $a_k$ and on the size of $\max_j a_j$.
\end{rem}
We will see that the following is an immediate consequence of the proof of Theorem \ref{thm:Gbounds}.
\begin{cor}\label{cor:Gbound_m=1}
Suppose that we have $m(x) = 1$ in Theorem \ref{thm:Gbounds}, that is $A/2 \leq x \leq 3A/2$. Then $G(x)$ in \eqref{gdef} satisfies
\begin{equation*}
    	\frac{2}{\pi}e^{-g(a_k)}\left(1 - \frac{2}{3A}\right)\left(1- \frac{1}{A}\right)^2 \, \frac{ \mathrm{dist}(x,\cu)}{|x|}\, \,\leq\, |G(x) |\,\, \leq\,\, \frac{14A}{9}e^{g(a_k)} \frac{1}{|x|}, 
\end{equation*}
where 
\begin{equation}\label{eq:ga_k}
    g(a_k) = \frac{3.3}{a_k}+\frac{1}{80\log a_k}+\frac{1}{400}+\frac{2}{a_k^2} \cdot
\end{equation}
\end{cor}

Before we embark on the proof of Theorem \ref{thm:Gbounds}, we establish two preliminary results. The first concerns the size of $A=2|e_kc_k|^{-1}$. 
\begin{lem}
Let $1\leq k \leq \ell.$	We have 
	\begin{equation} \label{ckekformula}
	\frac{1}{|c_ke_k|} = a_k + \frac{p_{\ell}(\alpha_{\sigma_k})}{q_{\ell}(\alpha_{\sigma_k})} + \frac{p_{\ell}(\alpha_{\tau_k})}{q_{\ell}(\alpha_{\tau_k})} - \frac{2b}{q_{\ell}(\alpha_{\tau_k})}\,  \cdot
	\end{equation}
\end{lem}
\begin{proof} By \eqref{qlplusone} we obtain
	\begin{eqnarray*}
		\frac{1}{|c_ke_k|} = \frac{q_{\ell+1}(\alpha_{\tau_k})+p_{\ell}(\alpha_{\tau_k})-2b}{q_{\ell}(\alpha_{\tau_k})} 
	= a_k + \frac{p_{\ell}(\alpha_{\sigma_k})}{q_{\ell}(\alpha_{\sigma_k})} + \frac{p_{\ell}(\alpha_{\tau_k})}{q_{\ell}(\alpha_{\tau_k})} - \frac{2b}{q_{\ell}(\alpha_{\tau_k})} \cdot
	\end{eqnarray*}
\end{proof}

\begin{cor}\label{cor:sizeA} For any $k=1, 2,\ldots, \ell $ we have 
	\begin{equation*} 
	a_k <\frac{1}{|c_ke_k|} < a_k +2 \, . 
	\end{equation*}
\end{cor}
By Corollary \ref{cor:sizeA}, it immediately follows that 
\begin{equation}\label{eq:sizeA}
    2a_k < A < 2(a_k+2) .
\end{equation}

The second result concerns the perturbations $\delta_t$. We state it without proof, as it is an easy consequence of \cite[Corollary 3]{pinner}.
\begin{lem}\label{lem:discr}
Let $\alpha=[0;a_1, a_2, \ldots]$ be an irrational with bounded continued fraction coefficients. Then for any fixed $n \in \mathbb{N}$, we have 
\begin{equation*}
    \Big| \sum_{t=n+1}^{n+N}\!\Big(\frac{1}{2} - \left\{ t\alpha\right\}\Big)\Big| \leq \log N \left( \frac{a}{8\log a} + 6\right)+ \frac{a}{8}+\frac{23}{4},
\end{equation*}
for all $N\geq 1$, where $a = \max_j a_j$.
\end{lem}
Recall from \eqref{eq:Aanddelta} that $\delta_t = 1-2\{t\alpha_{\sigma_k}\}$, and thus by Lemma \ref{lem:discr} it follows that 
\begin{equation}\label{eq:sumdeltas}
    \left| \sum_{t=n+1}^{n+N} \delta_t \right| \leq \min \left\{ N, \log N \left( \frac{a_k}{4\log a_k} + 12\right)+ \frac{a_k}{4}+\frac{23}{2} \right\},
\end{equation}
where $a_k = \max_j a_j$ for the quadratic irrational $\alpha$. Note that 
\begin{equation*}
    \log N \left( \frac{a_k}{4\log a_k} + 12\right)+ \frac{a_k}{4}+\frac{23}{2} \geq 12 \log N + \frac{23}{2} > N \quad \text{ for all } \quad N \leq 60, 
\end{equation*}
regardless of the value of $a_k$. Accordingly, we will use the bound $\sum_{n+1}^{n+N} \delta_t \leq N$ whenever $N \leq 60$.

We are now equipped to prove Theorem \ref{thm:Gbounds}.
\begin{proof}[Proof of Theorem \ref{thm:Gbounds}]
Since $G$ in \eqref{gdef} is an even function, it suffices to consider $x\geq0$. Let $m=m(x)\geq 0$ be the non-negative integer satisfying
\begin{equation*}
 | x -Am |\, = \, \min \left\{ |x-An| : n=0,1,2,\ldots  \right\},  
\end{equation*}
and let us first assume that $m\geq 1$, meaning that $x\geq A/2$. Excluding the case of the golden ratio, we may safely assume that $a_k = \max_j a_j \geq 2$, and thus by \eqref{eq:sizeA} we have $A\geq 4$. It follows that 
\begin{equation}\label{minimum}
|x- Am| \leq \frac{A}{2}, \qquad |x-Am-\delta_m|\leq \frac{A}{2} + 1 \leq \frac{3A}{4},
\end{equation}
and 
\begin{equation*} 
|x-At| \geq \frac{A}{2} \quad \text{ and } \quad |x-At-\delta_t| \geq   \frac{A}{2} -1 \geq  \frac{A}{4} ,
\end{equation*}
for any $t\neq m$.

The function $G$ may be split into three products
$G(x)=\Pi_1(x)\Pi_2(x)\Pi_3(x),$ where 
	\begin{eqnarray*}
		\Pi_1(x) &=& 1 - \frac{x^2}{(Am+\delta_m)^2},\\
		\Pi_2(x) &=& \mathop{\prod_{t\geq 1}}_{t\neq m}\left(  1- \frac{x^2}{(At+\delta_t)^2}\right) \left(1- \frac{x^2}{(At)^2} \right)^{-1},\\
		\Pi_3(x) &=& \mathop{\prod_{t\geq 1}}_{t\neq m}\left( 1 - \frac{x^2}{(At)^2}\right).
	\end{eqnarray*}
For the first product, we observe that 
\begin{equation*} 
\Pi_1(x) = 1 - \frac{x^2}{(Am+\delta_m)^2} = \frac{(Am+\delta_m-x)(Am+\delta_m+x)}{(Am+\delta_m)^2},
\end{equation*}
and thus by \eqref{minimum} we get
\begin{equation} \label{pi1bound}
	\text{dist}(x,\mathcal{U})\frac{Am+\delta_m+x}{(Am+\delta_m)^2} \leq |\Pi_1(x)| \leq \frac{3A}{4}\frac{Am + \delta_m +x}{(Am+\delta_m)^2} \, \cdot
\end{equation}
We then consider the second factor $\Pi_2(x)= \prod\limits_{t\neq m}Q_t(x)$, where
$$  Q_t(x) = \left(  1- \frac{x^2}{(At+\delta_t)^2}\right) \left(1- \frac{x^2}{(At)^2} \right)^{-1}, \quad t\geq 1 .  $$
We have 
\begin{eqnarray*}
		Q_t(x) &=& \frac{1- \dfrac{x}{At+\delta_t} }{1 - \dfrac{x}{At} } \cdot \frac{1+ \dfrac{x}{At+\delta_t}}{1+ \dfrac{x}{At }}\,\, = \,\, \frac{1+ \dfrac{\delta_t}{At-x}}{1+ \dfrac{\delta_t}{At}} \cdot \frac{1+\dfrac{\delta_t}{At+x}}{1+ \dfrac{\delta_t}{At}} \\[1ex]
		&=& \exp\left\{\log\left( 1+ \dfrac{\delta_t}{At-x} \right) +\log\left(  1+\dfrac{\delta_t}{At+x} \right) -2\log\left(1+ \frac{\delta_t}{At}\right)  \right\} .
	\end{eqnarray*}
Thus if we employ the inequality
\begin{equation*}  
x - x^2  \, < \, \log (1+x)  \,< \,x  \qquad \text{ for all }  x > -\frac12  
\end{equation*}
we obtain
\begin{eqnarray} \label{44}
	\quad \exp\left\{\delta_t s_t - \delta_t^2 \left( \frac{1}{(At-x)^2}+\frac{1}{(At+x)^2}\right)  \right\} < Q_t(x)  < \exp\left\{  \delta_t s_t + \frac{2\delta_t^2}{ (At)^2}    \right\}    , 	
\end{eqnarray}
where we define
\begin{equation}\label{eq:defsn}
	s_t = \frac{1}{At-x} + \frac{1}{At+x} -\frac{2}{At} \cdot
\end{equation}
The factors in \eqref{44} contributing significantly to $\Pi_2(x)= \prod_{t \neq m} Q_t$ are the first-order terms $\delta_t s_t$. For the second-order terms, we observe that the contribution on the left hand side is larger (in absolute value) than that on the right hand side, and a straightforward calculation verifies that 
	$$ \sum_{t\neq m} \left( \frac{1}{(At+x)^2} + \frac{1}{(At-x)^2} \right) < \frac{8}{A^2}.$$
We thus conclude that 
\begin{equation}\label{eq:Pi2max}
	\Pi_2(x) = \prod_{t\neq m} Q_t(x) = \exp \left( \sum_{t \neq m} \delta_t s_t + E \right),
\end{equation}
where $s_t$ is given in \eqref{eq:defsn} and $|E| < \dfrac{8}{A^2} $. 

It remains to find an appropriate bound for 
\begin{equation*}
  \left|\sum_{t\neq m} \delta_ts_t \right| \leq \left| \sum_{t<m} \delta_t s_t \right| + \left| \sum_{t>m} \delta_t s_t \right|.
\end{equation*}
We first consider the final term on the right hand side above. Summation by parts yields
\begin{equation*}
    \sum_{t=m+1}^{m+M} \delta_t s_t = s_{m+M} \sum_{t=m+1}^{m+M} \delta_t + \sum_{t=m+1}^{m+M-1} (s_{t}-s_{t+1}) \sum_{k=m+1}^t  \delta_k,
\end{equation*}
for any $M\geq 1$. We observe that 
\begin{equation*}
    |s_{t+1}-s_t| \leq \frac{1}{At-x}-\frac{1}{A(t+1)-x} \leq \frac{1}{A(t-m-\frac{1}{2})(t-m+ \frac{1}{2})},
\end{equation*}
and \begin{equation*}
    |s_{m+M}| \leq \frac{1}{A(m+M)-x} \leq \frac{1}{A(M-\frac12)} \cdot 
\end{equation*}
Combining this with \eqref{eq:sumdeltas}, we find that 
\begin{equation*}
    \left| \sum_{t=m+1}^{m+M} \delta_t s_t \right| \leq \varepsilon(M) + \sum_{t=m+1}^{m+M-1} \frac{\min \{ t-m, \, K\log (t-m)+C\}}{A\left((t-m)^2-\frac14 \right)},
\end{equation*}
where $K=12+a_k/(4\log a_k)$, $C=23/2 + a_k/4$, and where $\varepsilon(M) \to 0$ as $M \to \infty$. Letting $M \to \infty$, we thus find 
\begin{eqnarray*}
 \left| \sum_{t>m} \delta_t s_t \right| &\leq& \frac{1}{A} \left( \sum_{t=1}^{60} \frac{t}{(t^2-\frac14 )} + K \sum_{t=61}^{\infty} \frac{\log t}{(t^2-\frac14)} + C \sum_{t=61}^{\infty} \frac{1}{(t^2-\frac14)}  \right) \\
 &\leq& \frac{1}{A} \left( 5.1 + 0.1K + 0.02 C \right),
\end{eqnarray*}
and inserting values of $C$ and $K$, and recalling that $A>2a_k$, we get
\begin{equation*}
    \left| \sum_{t>m} \delta_t s_t \right| \leq \frac{3.3}{a_k}+\frac{1}{80 \log a_k} + \frac{1}{400} \cdot
\end{equation*}
By an analogous argument, one can show that 
\begin{eqnarray*}
     \left| \sum_{t<m} \delta_t s_t \right| &\leq& \frac{10.4}{a_k}+ \frac{3}{80 \log a_k} + \frac{3}{400},
\end{eqnarray*}
and thus combined we have
\begin{equation*}
    \left| \sum_{t \neq m} \delta_ts_t \right| \leq \frac{13.7}{a_k} + \frac{1}{20\log a_k}+ \frac{1}{100} \cdot
\end{equation*}
Inserting this in \eqref{eq:Pi2max}, we arrive at 
\begin{equation}\label{eq:Pi2finalbound}
e^{-f(a_k)} \leq \left| \Pi_2(x) \right| \leq e^{f(a_k)},
\end{equation}
with $f$ defined as in \eqref{eq:fa_k}.

Finally, we observe that
$$ \Pi_3(x) = \frac{A^2m^2}{(Am-x)(Am+x)} \cdot \frac{\sin(\pi xA^{-1})}{\pi x A^{-1}} $$ 
and since $ |\pi xA^{-1} - \pi m| \leq \dfrac{\pi}{2}$ we get
$$ \frac{2}{A} \, \leq \, \left| \frac{\sin(\pi xA^{-1}) }{Am-x }\right|  \, \leq \, \frac{\pi}{A} \, \cdot   $$ 
This implies that 
\begin{equation} \label{pi3bounds}
	\frac{2A^2m^2}{\pi(Am+x)x}\, \leq\, |\Pi_3(x) |\, \leq \, \frac{A^2m^2}{(Am+x)x} \, .
\end{equation}

Combining the bounds \eqref{pi1bound}, \eqref{eq:Pi2finalbound} and \eqref{pi3bounds}, we find that 
\[ |G(x)| \geq \frac{2}{\pi} e^{-f(a_k)} \cdot \mathrm{dist}(x,\mathcal{U}) \cdot \frac{1}{|x|} \cdot \frac{(Am)^2}{(Am+\delta_m)^2} \cdot \frac{Am+ \delta_m + x}{Am + x},\]
and
	\[|G(x)| \leq \frac{3A }{4}e^{f(a_k)} \cdot \frac{1}{|x|} \cdot \frac{(Am)^2}{(Am+\delta_m)^2} \cdot \frac{Am+\delta_m + x}{Am + x}\, \cdot \]
The common factor 
\[\frac{(Am)^2}{(Am+\delta_m)^2} \cdot \frac{Am+\delta_m + x}{Am + x} = \left( 1-\frac{\delta_m}{Am+\delta_m}\right)^2 \left(1+\frac{\delta_m}{Am+x} \right)\]
will necessarily tend to $1$ as $m(x) \to \infty$. Only the rate of convergence from below will be important to us. We therefore apply the rough upper bound 
\[\frac{(Am)^2}{(Am+\delta_m)^2} \cdot \frac{Am+\delta_m + x}{Am + x}  \leq \left( \frac{4}{3}\right)^2 \left( \frac{7}{6}\right),\]
and the more precise lower bound 
\[\frac{(Am)^2}{(Am+\delta_m)^2} \cdot \frac{Am+\delta_m + x}{Am + x} \geq \left( 1-\frac{1}{Am}\right)^2 \left( 1-\frac{2}{3Am} \right) .\]
Inserting these bounds in the inequalities for $|G(x)|$ completes the proof of Theorem \ref{thm:Gbounds} in the case $|x|\geq A/2$.

Finally, we consider the case $0\leq x <A/2$. As an upper bound, we use
\[|G(x)| \leq 1.\]
For the lower bound, we again split $G$ into the subproducts $\Pi_1$, $\Pi_2$ and $\Pi_3$. Note that in this case, the product $\Pi_1$ is empty, and the product $\Pi_3$ is simply a sinc function bounded by 
	\[\frac{2}{\pi} \leq \Pi_3(x) \leq 1 .\]
For the product $\Pi_2$, we may use the bound \eqref{eq:Pi2finalbound} established for the case $x\geq A/2$ (in fact we can do better, as will be argued below). Combined we get
\begin{equation*}
\frac{2}{\pi} e^{-f(a_k)} \leq |G(x)| \leq 1,
\end{equation*}
and this completes the proof of Theorem \ref{thm:Gbounds}.
\end{proof}
\begin{proof}[Proof of Corollary \ref{cor:Gbound_m=1}]
Retracing the proof of Theorem \ref{thm:Gbounds}, we arrive at \eqref{eq:Pi2max}, and note that for $m(x)=1$ we have
\begin{equation*}
    \Pi_2(x) = \prod_{t \neq m} Q_t(x) = \exp \left( \sum_{t>m} \delta_ts_t + \frac{8}{A^2}\right) \leq \exp \left( \frac{3.3}{a_k} + \frac{1}{80 \log a_k}+ \frac{1}{400} + \frac{2}{a_k^2} \right),
\end{equation*}
since the product $\sum_{t<m} \delta_ts_t$ is empty. Apart from this, the proof remains unchanged.
\end{proof}
\begin{rem}\label{rem:Gsmallx}
Note that this bound on $\Pi_2(x)$ is clearly also valid for $x\leq A/2$. Thus, the lower bound on $G(x)$ in \eqref{eq:Gsmallx} may be improved to 
\begin{equation*}
    \frac{2}{\pi} e^{-g(a_k)} \leq |G(x)| \leq 1, \quad x\leq A/2,
\end{equation*}
with $g$ given in \eqref{eq:ga_k}.
\end{rem}

\begin{rem}
The lower bound for $G(x)$ in \eqref{Grate} is used in the following sections to determine an upper bound for the constants $C_k.$ In turn, this upper bound gives a threshold value $K> 1$ such that  $a_k = \max_{1\leq i \leq \ell}a_i \geq K$ implies that $C_k<1.$ \par  We believe that an improvement for the threshold value $K=23$ in Theorem \ref{thm:ak23} might be obtained as follows: in the estimates for $\Pi_2(x),$ we have bounded $\Big| \sum\limits_{t\neq m(x)}\delta_t s_t(x)\Big|$ from above uniformly for all $x>0$. In the proof of Theorem \ref{thm:ak23} we are actually interested in the quantity $G(1)G(3)\cdots G(2c-1),$ which means that we need a bound for the product $\prod_{x=1}^c \Pi_2(2x-1)$. Rather than using a uniform bound for all terms in this product, we could seek an upper bound for the double sum \[  \Big|\sum_{1\leq x\leq c}\hspace{-3mm} \sum_{\substack{t\geq 1 \\ t\neq m(2x-1)}}\hspace{-4mm} \delta_t s_t(2x-1)\Big|. \] 
This would take the specific range of $x$-values into account, and possibly provide a substantial improvement in the lower bound on $\prod_{x=1}^c \Pi_2(2x+1)$. We leave this task to the interested reader.
\end{rem}

\section{Upper bounds for the constants $C_k$\label{sec:sizeCk}}
With Theorem \ref{thm:Gbounds} established, let us now revisit Corollary \ref{cor:C_k} and carefully analyse the expressions for $C_k$ provided in \eqref{cvalueseven} and \eqref{cvaluesodd}. Recall that we have to differentiate between the case of even and odd period length $\ell$. To ease the analysis, it will be useful to make an assumption on the size of $\max_j a_j$ given a quadratic irrational $\alpha=[0; \overline{a_1, \ldots , a_{\ell}}]$. Presuming a priori that we cannot do better for general $\ell$ than for the $\ell=1$ case (see \cite[Theorem 6]{ATZ20}), we assume throughout this section that $\max_j a_j \geq 6$. 

We will show the following. 
\begin{thm}\label{thm:boundCk}
Let $\alpha =[0;\overline{a_1, \ldots, a_\ell}]$ for some even period length $\ell$, and assume $a_k = \max_j a_j \geq 6$. Then the limit $C_k = \lim_{m \to \infty} P_{\ell m+k}(\alpha)$ obeys the bound
\begin{equation*} 
C_{k}^{\frac{c-2}{c}} \leq \frac{\pi}{2a_k} e^{1+f(a_k)} \left(200e^{2.4}c^2\right)^{\frac{1}{c}} \cdot 2^{\frac{1}{q_{\ell}}} \cdot \left(\frac{a_k^{\frac52}}{e} \right)^{\frac{1}{a_k}}.
\end{equation*}
where $q_\ell = q_{\ell}(\alpha_{\sigma_k}),$ $c$ is defined in \eqref{abc} and we recall from \eqref{eq:fa_k} that 
\begin{equation*}
    f(a_k) \leq \frac{13.7}{a_k}+0.1, \quad a_k\geq 6.
\end{equation*}
\end{thm}
Throughout this section, when there is no danger of confusion, we shall write for abbreviation $q_\ell = \ql$ and $q_{\ell+1} = \qlplusone.$ 

For the special case $q_{\ell}=1$, we have an improved bound; note that this only occurs if $\ell=2$ and $\alpha=[0; \overline{a_1, a_2}]$ with either $a_1=1$ or $a_2=1$.
\begin{cor}\label{cor:boundCkq=1}
Let $\alpha=[0;\overline{a_1 , a_2}]$, and assume $a_k = \max \{a_1, a_2\} \geq 6$ and $\min \{a_1, a_2\}=1$. Then the limit $C_k = \lim_{m \to \infty} P_{2m+k}(\alpha)$ obeys the bound
\begin{equation*} 
    C_k^{\frac{c-2}{c}} \leq \frac{\pi}{a_k}e^{1+g(a_k)} \left( 6.2(a_k+2)^4\right)^{\frac{1}{a_k+2}}, 
\end{equation*}
where we recall from \eqref{eq:ga_k} that 
\begin{equation*}
    g(a_k) \leq \frac{3.3}{a_k}+ 0.1 , \quad a_k \geq 6 .
\end{equation*}
\end{cor}

For the odd period case, we will establish the following.
\begin{thm}\label{thm:boundCkodd}
Let $\alpha =[0;\overline{a_1, \ldots, a_\ell}]$ for some odd period length $\ell$, and assume $a_k = \max_j a_j \geq 6$. Then the limit $C_k = \lim_{m \to \infty} P_{\ell m+k}(\alpha)$ obeys the bound
\begin{equation*} 
C_{k} \leq \frac{\pi}{2a_k} e^{1+f(a_k)} \left(40c^{\frac32}\right)^{\frac{1}{c}} \cdot 2^{ \frac{1}{q_{\ell}}} \cdot a_k^{\frac{5}{2a_k}},
\end{equation*}
where $f$ is given in \eqref{eq:fa_k}.
\end{thm}

\subsection{Bounding $C_k$ when $\ell \equiv 0 \pmod 2$}
Considering first even period lengths $\ell$, we recall from \eqref{cvalueseven} that
\begin{eqnarray} 
C_k^{c-2} &=& \frac{1+a^2}{c!}\left| \frac{G(1)^{c-1}G(1+2a^2) }{ G(1)G(3)G(5)\cdots G(2c-1)}\right| , \label{CkBound1}
\end{eqnarray}
with $a$ and $c$ as given in \eqref{abc}, and the function $G$ defined in \eqref{gdef}. The term $|G(1)|$ in \eqref{CkBound1} is bounded above by $1$, and the term $G(1+2a^2)$ can be bounded by the upper bound in Theorem \ref{thm:Gbounds}. Keeping the expression for $C_k^{c-2}$ in mind, we will rather give a bound for $(1+a^2)G(1+2a^2)$. By Theorem \ref{thm:Gbounds} we have
\begin{equation}\label{eq:boundGbig}
    (1+a^2)G(1+2a^2) \leq \frac{14A}{9}\cdot e^{f(a_k)} \cdot \frac{(1+a^2)}{(1+2a^2)} \leq \frac{4}{5}Ae^{f(a_k)},
\end{equation}
where for the last inequality we have used that $a\geq 5$ whenever $a_k \geq 6$.

Now let us find a lower bound on 
\begin{equation} \label{Gvalues}
G(1)G(3)\cdots G(2c-1) = \prod_{s=0}^{c-1}G(2s+1).
\end{equation}
In view of Theorem \ref{thm:Gbounds}, some of the factors of \eqref{Gvalues} will be bounded using \eqref{Grate} while others will be bounded using \eqref{eq:Gsmallx}. Since $2a_k <A < 2(a_k+2)$, the integers $j$ satisfying $j < \frac{A}{2}$ are $j=1,2,\ldots, a_k$ and possibly also $a_k+1$. Thus the factors of \eqref{Gvalues} with $0\leq s \leq  \left\lfloor \frac{a_k}{  2}\right\rfloor -1$ will be bounded using \eqref{eq:Gsmallx} and those with $s\geq \lfloor \frac{a_k}{2}\rfloor+1$  will be bounded using \eqref{Grate}. We get 
\begin{equation}\label{Gsmallarguments}
\prod_{s = 0}^{ \lfloor \frac{a_k}{2} \rfloor - 1} G(2s+1) \geq \left(\frac{2}{\pi} e^{-f(a_k)}\right)^{  \lfloor \frac{a_k}{2} \rfloor } 
\end{equation}
and 
\begin{equation}\label{Gbigarguments}
\prod_{s= \lfloor \frac{a_k}{2} \rfloor+1}^{c-1} G(2s+1) \geq  
\prod_{s= \lfloor  \frac{a_k}{2} \rfloor+1}^{c-1} \frac{2}{\pi}e^{-f(a_k)} \left( 1 - \frac{2}{ 3Am_{s}} \right) \left( 1 - \frac{1}{ A m_{s}}\right)^2 \cdot \frac{\mathrm{dist}(2s+1,\mathcal{U})}{2s+1},
\end{equation}
with $\mathcal{U}= (u_k(t))_{t=1}^{\infty}$ given in \eqref{ukt} and $m_s \in \mathbb{N}$ as defined in Theorem \ref{thm:Gbounds}.

The factor $G(a_k+1)$ appears in \eqref{Gvalues} only when $a_k$ is even and $s=a_k/2$. For this factor, it is not clear which of the two bounds \eqref{Grate} and \eqref{eq:Gsmallx} apply. However, under the restriction that $a_k\geq 6$, we clearly have $\mathrm{dist}(a_k+1, \mathcal{U})>1$, and by combining the bounds \eqref{Grate} and \eqref{eq:Gsmallx} (keeping all terms that are below $1$), we get the universal bound
\begin{equation*}
G(2 \lfloor \tfrac{a_k}{2} \rfloor+1) \geq \frac{2}{\pi} e^{-f(a_k)} \left( 1 - \frac{2}{3 A } \right) \left( 1 - \frac{1}{A}\right)^2  \frac{1}{  2\lfloor \tfrac{a_k}{2} \rfloor + 1 },
\end{equation*}
which holds regardless of the parity of $a_k$, and of whether $a_k+1<A/2$ or $a_k+1\geq A/2$. Combining this bound with \eqref{Gsmallarguments} and \eqref{Gbigarguments}, we finally get
\begin{equation}\label{eq:Gallvals}
\prod_{s=0}^{c-1} G(2s+1) \geq \left( \frac{2}{\pi} e^{-f(a_k)}\right)^c \cdot \frac{\Pi_1 \cdot \Pi_2}{\Pi_3}, 
\end{equation}
where 
\begin{equation} \label{firstproduct}
\Pi_1   =   \prod_{s=\lfloor \frac{a_k}{2} \rfloor}^{c-1} \left( 1 - \frac{2}{3 A m_{s}} \right) \left( 1 - \frac{1}{A m_{s}}\right)^2 , 
\end{equation}
\begin{equation} \label{product}
\Pi_2\quad  =  \prod_{s = {\lfloor \frac{a_k}{2} \rfloor}+1}^{c-1} \mathrm{dist}(2s+1,\mathcal{U})  \, ,
\end{equation}
and 
\begin{equation}\label{denomproduct}
\Pi_3 = \prod_{s= \lfloor \frac{a_k}{2} \rfloor}^{c-1} (2s+1) . 
\end{equation}

We proceed by bounding the three product terms $\Pi_1$, $\Pi_2$ and $\Pi_3$ separately. In the following, we will make use of the inequalities
\begin{gather}
 \sqrt{2\pi n }\, \left(\frac{ n}{e} \right)^n  \, \leq \, n ! \, \leq \,  e \sqrt{n}  \left(\frac{ n}{e} \right)^n  , \nonumber \\[1ex]
\sqrt{2\pi n }\, \left(\frac{2n}{e} \right)^n  \, \leq \, (2n)!! \, \leq \, e \sqrt{n}  \left(\frac{2n}{e} \right)^n   , \label{factorials}  \\[1ex]
\frac{\sqrt{4\pi}}{e}  \, \left( \frac{2n}{e}\right)^{n} \, \leq \,  (2n-1)!! \, \leq \, \frac{e}{ \sqrt{\pi}} \left(\frac{2n}{e} \right)^n    \nonumber
\end{gather}
which are valid for all $n \in \mathbb{N}$. 

Starting with the first and simplest of the three, we observe that 
\begin{equation*}
\Pi_1 \geq \prod_{s=\lfloor \frac{a_k}{2} \rfloor}^{c-1} \left( 1 - \frac{1}{ A m_{s}}\right)^3.
\end{equation*}
Recall that $m_s$ is the unique positive integer for which $|Am_s-(s+1)|$ is minimized. As $s$ runs through the values $\lfloor a_k/2\rfloor, \ldots , c-1$, the integer $m_s$ runs through the values $1, \ldots , q_\ell$, and each integer occurs at most $a_k+1$ times. It follows that 
\begin{equation*}
\Pi_1 \geq   \prod_{m=1}^{q_\ell} \left( 1-\frac{1}{Am} \right)^{3(a_k+1)}.
\end{equation*}
Using that 
\begin{equation*}
 x - x^2 \, < \, \log (1+x)  \,< \,x   \qquad \text{ for all }  x > - \frac12 
\end{equation*}
it is straightforward to show that
\begin{equation*}
\prod_{m=1}^{q_\ell} \left( 1-\frac{1}{Am} \right) \geq \exp \Big(\!\! -\frac{1}{A}\left(1+\log q_\ell\right) - \frac{\pi^2}{6A^2}\Big) \geq \exp \Big(\!\!-\frac{1}{A}(1.14 + \log q_\ell)\Big),
\end{equation*}
where for the final inequality we have used that $a_k\geq 6$, and thus $A>12$. Inserting this in the expression for $\Pi_1$ above, we get
\begin{equation}\label{eq:Pi1bd}
\Pi_1 \geq \left( \frac{1}{e^{1.14}q_\ell}\right)^{\frac{3(a_k+1)}{A}} \geq \left( \frac{1}{e^{1.14}q_\ell}\right)^{\frac{3(a_k+1) }{2a_k}} \geq \frac{1}{e^2q_\ell^2}.
\end{equation}

Let us now consider $\Pi_3$, for which we need to find an upper bound. From \eqref{factorials} it follows that 
\begin{equation*}
    \Pi_3 = \frac{(2c-1)!!}{\left(2\lfloor \frac{a_k }{2} \rfloor -1\right)!!} \leq \frac{e^2}{2\pi} \left( \frac{2c}{e}\right)^c \left( \frac{2 \lfloor \frac{a_k }{2} \rfloor}{e} \right)^{-\lfloor  \frac{a_k }{2}\rfloor} .
\end{equation*}
We observe that 
\begin{equation}\label{eq:factrick}
\left( \frac{2\lfloor \frac{a_k }{2} \rfloor}{e}\right)^{\lfloor \frac{a_k }{2} \rfloor} \geq \left( \frac{a_k-1}{e}\right)^{ \frac{a_k-1 }{2}} \geq \frac{1}{2} \left( \frac{e}{a_k} \right)^{\frac12} \left( \frac{a_k}{e}\right)^{ \frac{a_k }{2}},
\end{equation}
where in the last step we have used that $(1-1/n)^{\frac{n}{2}}   \geq \frac12$ for $n\geq 2$. Inserting this in the bound above we find that
\begin{equation}\label{eq:finalPi3bd}
\Pi_3 \leq \frac{e^{\frac32}\sqrt{a_k}}{\pi} \left(\frac{e}{a_k}\right)^{\frac{a_k }{2}} \left( \frac{2c}{e}\right)^{c} .
\end{equation}

The estimation of the product $\Pi_2$ is far more intricate. When we analyse the size of the factors comprising $\Pi_2,$ the numbers 
\begin{equation} \label{Rtdefinition}
R_t \, = \,  \left\{\frac{t\pltau}{\qltau} \right\} + t\left(\frac{\pl}{\ql}-\alpha_{\sigma_k} \right)- \frac{2bt}{\ql}, \quad   t=1,\ldots, q_{\ell}-1 
\end{equation}
appear naturally. The size of the product of $\|R_t\|$ and its relation with $\Pi_2$ is stated in the following  lemmas,  the proofs of which are postponed for Section \ref{sec:pi2lemmas}.
\begin{lem}\label{lem:pi2two}
When $\ell$ is even, the numbers $R_t, t=1,\ldots, q_\ell -1$ given in \eqref{Rtdefinition} satisfy 
\begin{equation} \label{eq:boundRt}
\prod_{t=1}^{q_\ell-1} \|R_t\| \geq \frac{\sqrt{q_\ell}}{2(2e)^{q_\ell+1}} \cdot
\end{equation}
\end{lem}

\begin{lem}\label{lem:pi2one}
The product $\Pi_2$ defined in \eqref{product} is bounded below by 
\begin{equation}\label{eq:firstpi2bd}
\quad\Pi_2 \geq \frac{2}{5c}\cdot  2^{c- \lfloor \frac{a_k}{2}\rfloor}\cdot \left\lfloor \frac{a_k}{2} \right\rfloor !  \cdot [(a_k-1)!]^{q_\ell-1}  \cdot\prod_{t=1}^{q_\ell-1} \|R_t\| . 
\end{equation}
\end{lem}
We now find a bound for $\Pi_2$ using Lemmas \ref{lem:pi2two} and \ref{lem:pi2one}. By \eqref{factorials} we have
\begin{equation*}
\Big\lfloor \frac{a_k}{2} \Big\rfloor ! \geq \left( \frac{a_k}{2e}\right)^{\frac{a_k }{2}} \quad \text{ and } \quad (a_k-1)! \geq \frac{2}{\sqrt{a_k}} \left( \frac{a_k}{e} \right)^{a_k} .
\end{equation*}
The former is obvious when $a_k$ is even, and follows by an argument similar to \eqref{eq:factrick} when $a_k$ is odd and greater than $6$. Inserting these bounds in \eqref{eq:firstpi2bd} and employing Lemma \ref{lem:pi2two}, we get
\begin{equation}\label{eq:pi2finalbd}
\Pi_2 \geq \frac{\sqrt{a_kq_\ell}}{20ec}\cdot 2^{c-a_k} \left( \frac{1}{e\sqrt{a_k}}\right)^{q_\ell} \cdot \left( \frac{a_k}{e}\right)^{a_k(q_\ell-\frac12)}.
\end{equation}
Combining the bounds established for $\Pi_1$, $\Pi_2$ and $\Pi_3$, we conclude as follows.
\begin{lem}\label{lem:gprodlower}
Let $G$ be given in \eqref{gdef}. Under the assumption that $a_k\geq 6$ we have
\begin{equation*}
\prod_{s=0}^{c-1} G(2s+1) \geq \left( \frac{2a_k}{\pi e^{f(a_k)}c}\right)^c \cdot \frac{2^{-a_k}\pi}{20e^{\frac92}q_{\ell}^{\frac32}c} \cdot \left( \frac{e}{a_k^{5/2}}\right)^{q_{\ell}}.
\end{equation*}
\end{lem} 
\begin{proof}
Inserting the bounds established for $\Pi_1$, $\Pi_2$ and $\Pi_3$ in \eqref{eq:Pi1bd}, \eqref{eq:pi2finalbd} and \eqref{eq:finalPi3bd} into \eqref{eq:Gallvals}, we get
\begin{align*}
\prod_{s=0}^{c-1} G(2s+1) &\geq \left(\frac{2}{\pi c}e^{1-f(a_k)} \right)^c \cdot \frac{2^{-a_k} \pi}{20e^{\frac92}q_{\ell}^{\frac32} c} \left(\frac{1}{e\sqrt{a_k}} \right)^{q_{\ell}} \left(\frac{a_k}{e} \right)^{a_kq_{\ell}} \\
&\geq \left(\frac{2a_k}{\pi e^{f(a_k)} c} \right)^c \cdot \frac{2^{-a_k} \pi}{20e^{\frac92}q_{\ell}^{\frac32} c} \left(\frac{e}{a^{\frac52}} \right)^{q_{\ell}},
\end{align*}
where for the last inequality we have used that $c<(a_k+2)q_{\ell}$.
\end{proof}
We are now fully equipped to prove Theorem \ref{thm:boundCk}.
\begin{proof}[Proof of Theorem \ref{thm:boundCk}]
We recall from \eqref{CkBound1} and \eqref{eq:boundGbig} that 
\begin{equation*}
C_k^{c-2} = \frac{1+a^2}{c!} \left| \frac{G(1)^{c-1}G(1+2a^2)}{\prod_{s=1}^{c-1} G(2s+1)}\right| \leq \frac{1}{c!} \cdot \frac{4Ae^{f(a_k)}}{5} \cdot \frac{1}{\left|\prod_{s=1}^{c-1} G(2s+1)\right|}.
\end{equation*}
Using Lemma \ref{lem:gprodlower} and the bound on $c!$ in \eqref{factorials}, we thus get
\begin{equation*}
    C_k^{c-2} \leq \frac{16Ae^{\frac92}q_{\ell}^{\frac32} \sqrt{c}e^{f(a_k)}}{\pi \sqrt{2\pi}} \left( \frac{\pi e^{1+f(a_k)}}{2a_k}\right)^c \cdot 2^{a_k} \cdot \left( \frac{a_k^{\frac52}}{e}\right)^{q_{\ell}}
\end{equation*}
Recall that $A< 2(a_k+2)$ and $q_{\ell} < c/a_k$. This allows us to bound the first term on the right hand side by
\begin{equation*}
    \frac{32e^{\frac92}}{\pi \sqrt{2\pi}} \, \frac{(a_k+2)}{a_k^{3/2}} e^{f(a_k)} c^2 \leq \frac{32e^{\frac92}}{\pi \sqrt{2\pi}} \, \frac{8}{6^{3/2}} e^{f(6)} c^2 \leq 200e^{2.4}c^2 .
\end{equation*}
Inserting this in the expression above, we get
\begin{equation*}
    C_k^{c-2} \leq \left( \frac{\pi e^{1+f(a_k)}}{2a_k}\right)^c \cdot (200e^{2.4}c^2) \cdot 2^{a_k} \cdot \left(\frac{a_k^{\frac52}}{e} \right)^{q_{\ell}} .
\end{equation*}
Raising both sides to the power $1/c$ and using again that $c>a_kq_{\ell}$ completes the proof of Theorem \ref{thm:boundCk}.
\end{proof}
\begin{proof}[Proof of Corollary \ref{cor:boundCkq=1}]
When $q_{\ell}=1$ we have $c=a_k+2$, and revisiting Lemma \ref{lem:gprodlower} we see that we can obtain an improved bound on $\prod G(2s+1)$ in this case. From Corollary \ref{cor:Gbound_m=1} and Remark \ref{rem:Gsmallx} it follows that 
\begin{equation*}
    \prod_{s=0}^{c-1} G(2s+1) \geq \left( \frac{2}{\pi}e^{-g(a_k)}\right)^c \cdot \frac{\Pi_1 \cdot \Pi_2}{\Pi_3},
\end{equation*}
with $\Pi_1$, $\Pi_2$ and $\Pi_3$ defined as in \eqref{firstproduct}--\eqref{denomproduct}, respectively, and $g(a_k)$ given in \eqref{eq:ga_k}. We keep the bounds for $\Pi_1$ and $\Pi_3$ established in \eqref{eq:Pi1bd} and \eqref{eq:finalPi3bd}, and note that the bound on $\Pi_2$ in \eqref{eq:firstpi2bd} simplifies to 
\begin{equation*}
    \Pi_2 \geq \frac{2}{5c} \cdot 2^{c-\lfloor \frac{a_k}{2} \rfloor} \cdot \Big\lfloor \frac{a_k}{2} \Big\rfloor ! \geq \frac{8}{5c} \left( \frac{a_k}{e}\right)^{\frac{a_k}{2}}.
\end{equation*}
Inserting all three bounds above, we get 
\begin{equation}\label{eq:Gboundimproved}
    \prod_{s=0}^{c-1} G(2s+1) \geq \frac{8\pi}{5e^{\frac{7}{2}} c\sqrt{a_k}}\left( \frac{e^{1-g(a_k)}}{\pi c}\right)^c \left( \frac{a_k}{e}\right)^{a_k} 
    \geq \frac{8 \pi}{5e^{\frac{3}{2}}c^{\frac{7}{2}}} \left(\frac{e^{-g(a_k)}a_k}{\pi c} \right)^c,
\end{equation}
where we have used that $c=a_k+2>a_k$.

Again we have that
\begin{equation*}
C_k^{c-2} = \frac{1+a^2}{c!} \left| \frac{G(1)^{c-1}G(1+2a^2)}{\prod_{s=1}^{c-1} G(2s+1)}\right| \leq \frac{1}{c!} \cdot \frac{4Ae^{f(a_k)}}{5} \cdot \frac{1}{\left|\prod_{s=1}^{c-1} G(2s+1)\right|}.
\end{equation*}
Inserting $A< 2(a_k+2)=2c$ and the improved bound \eqref{eq:Gboundimproved} on $\prod G(2s+1)$, we get
\begin{equation*}
   C_k^{c-2} \leq \frac{1}{c!} \cdot \frac{e^{f(a_k)+\frac32}c^{\frac92}}{\pi}\left( \frac{\pi e^{g(a_k)}c}{a_k}\right)^c \leq \frac{e^{f(a_k)+\frac32}}{\pi \sqrt{2\pi}} \cdot c^4 \cdot \left( \frac{\pi e^{1+g(a_k)}}{a_k}\right)^c,
\end{equation*}
 where for the last inequality we have used the lower bound on $c!$ in \eqref{factorials}. The proof is completed by bounding $e^{f(a_k)}$ by $e^{f(6)}$ and raising both sides to the power $1/c$. 
\end{proof}

\subsection{Bounding $C_k$ when $\ell \equiv 1 \pmod 2$} Now let us consider the case of odd period lengths $\ell$. By \eqref{cvaluesodd}, the constant $C_k$ is then given by 
 \begin{equation}\label{cest2}
  C_k^{c}  = \frac{G(1)^c}{\prod\limits_{s=1}^{c}\left|s-a\right| } \prod_{s=0}^{c-1}|G(2a -2s -1)|^{-1} ,
\end{equation}
where $a$ and $c$ are given in \eqref{abc} and $G$ is the function defined in \eqref{gdef}. Our goal is again to derive a bound for $C_k$ in the case when $k$ is the index such that $\max_j a_j = a_k$.

The assumptions that $\ell$ is odd and $a_k\geq 6$ necessarily imply that $c \geq 13;$ we make use of this inequality in the estimates that follow. The definitions of $a$ and $c$ in \eqref{abc} imply that 
$$   \frac{1}{a-c} \, =\, \frac{\sqrt{c^2+4} +c }{ 2 } < c + \frac{1}{13} \qquad \text{ and } \qquad  |a-s| \, > \, c-s, \quad s=1,\ldots, c-1 $$
and therefore
\begin{equation} \label{sminusa}
\prod\limits_{s=1}^{c}\left|s-a\right|^{-1} \leq \frac{c+ \frac{1}{13}}{(c-1)!} = \frac{c(c+\frac{1}{13})}{c!} \stackrel{\eqref{factorials}}\leq \frac{\sqrt{c}(c+\frac{1}{13})}{\sqrt{2\pi}}\Big(\frac{e}{c}\Big)^c .
\end{equation}

We now seek a lower bound for the product 
\begin{equation} \label{product_odd}
\prod\limits_{s=0}^{c-1}|G(2a -2s -1)| = \prod_{s=0}^{c-1} |G(2s+1 + 2(a-c))| = \prod_{s=0}^{c-1} |G(2s+1-2b)| .
\end{equation} 
We argue as in the case of even $\ell$, and use Theorem \ref{thm:Gbounds} to bound the terms of this product. We use \eqref{eq:Gsmallx} to bound the factors of \eqref{product_odd} with $0\leq s \leq \lfloor \frac{a_k}{2} \rfloor -1$, which gives 
\begin{equation*} 
\prod_{s = 0}^{ \lfloor \frac{a_k}{2} \rfloor - 1} |G(2s+1-2b)| \geq \left(\frac{2}{\pi e^{f(a_k)}} \right)^{\lfloor \frac{a_k}{2}\rfloor}. 
\end{equation*}
For the factors of \eqref{product_odd} corresponding to $\lfloor \frac{a_k}{2} \rfloor +1 \leq s \leq c-1,$ the bound \eqref{Grate} gives 
\begin{align*} 
\prod_{s=\lfloor \frac{a_k}{2}\rfloor+1}^{c-1} \hspace{-3mm}|G(2s+1-2b)| \geq & \hspace{-4mm}
\prod_{s= \lfloor \frac{a_k}{2} \rfloor+1}^{c-1} \hspace{-3mm}\frac{2}{\pi e^{f(a_k)}}\!\!\left( 1 - \frac{2}{3 A m_{s}} \right)\!\!\left(1 - \frac{1}{ A m_{s}}\right)^2 \frac{\mathrm{dist}\big(2s+1-2b,\mathcal{U}\big)}{2s+1-2b}  , 
\end{align*}
where the integers $m_s$ are defined in Theorem \ref{thm:Gbounds}. For the factor $ |G(2\lfloor \tfrac{a_k}{2}\rfloor +1-2b)| $ we use the bound 
 \begin{equation*}
|G(2 \lfloor \tfrac{a_k}{2} \rfloor+1 - 2b)| \geq \frac{2}{\pi} e^{-f(a_k)} \left( 1 - \frac{2}{3A} \right) \left( 1 - \frac{1}{A }\right)^2  \frac{1}{  2\lfloor \tfrac{a_k}{2} \rfloor + 1 - 2b} \cdot
\end{equation*}
Combining the estimates above, we obtain 
\begin{equation}\label{eq:Gallvals_odd}
\prod_{s=0}^{c-1} |G(2s+1-2b)| \geq \left( \frac{2}{\pi e^{f(a_k)}} \right)^c \cdot \frac{\Pi_1 \cdot \Pi_2'}{\Pi_3'}, 
\end{equation}
where $\Pi_1$ is the  product defined in \eqref{firstproduct}, while 
\begin{equation} \label{eq:pi23marked}
\Pi_2' = \!\!\!\prod_{s = {\lfloor \frac{a_k}{2} \rfloor}+1}^{c-1}\!\! \mathrm{dist}(2s+1-2b,\, \mathcal{U}) \quad \text{ and } \quad  \Pi_3'= \prod_{s=\lfloor \frac{a_k}{2} \rfloor}^{c-1} (2s+1-2b) .
\end{equation}

Let us first find a bound for $\Pi_3'$ by comparing it with $\Pi_3$ in \eqref{denomproduct}. Note that the bound \eqref{eq:finalPi3bd} on $\Pi_3$ does not depend on the parity of $\ell$. Since 
\[ -2b = \frac{4}{\sqrt{c^2+4}+c} < \frac{2}{c}, \]
we have 
\begin{eqnarray*}
\Pi_3' &\leq& \prod_{s=\lfloor \frac{a_k}{ 2} \rfloor}^{c-1} \Big(2s+1+\frac2c \Big) 
\,\, = \,\, \Pi_3 \cdot  \prod_{s=\lfloor\frac{a_k}{ 2} \rfloor}^{c-1}\Big(1 + \frac{2}{(2s+1)c}\Big).
\end{eqnarray*}
We find that 
\begin{align*}
\prod_{s=\lfloor \frac{a_k}{ 2} \rfloor}^{c-1}\Big(1 + \frac{2}{(2s+1)c}\Big) & = \exp\Bigg( \sum_{s=\lfloor \frac{a_k}{2}\rfloor}^{c-1}\log\Big(1 + \frac{2}{(2s+1)c}\Big)\Bigg) \\
& < \exp\Bigg(\frac{2}{c}\sum_{s=\lfloor \frac{a_k}{2}\rfloor}^{c-1}\frac{1}{2s+1}\Bigg) < \exp\Big(\frac{\log (c-1)}{c}\Big) < e^{\frac{1}{5}} ,
\end{align*}
so in view of \eqref{eq:finalPi3bd} we deduce that 
\begin{equation} \label{Pi_3_odd_estimate}
    \Pi_3'  \leq \frac{e^{\frac{17}{10}}\sqrt{a_k}}{\pi} \left(\frac{e}{a_k}\right)^{\frac{a_k}{ 2}} \left( \frac{2c}{e}\right)^{c}.
\end{equation}

As for the even period case, the estimation of $\Pi_2'$ is more elaborate. When analysing the factors in $\Pi_2'$, the numbers $R_t+b$ appear naturally, where $R_t$ is defined as in \eqref{Rtdefinition}. The size of products over $\|R_t+b\|$ and its relation to $\Pi_2'$ is stated in the lemmas below. Note that these are analogues of Lemmas \ref{lem:pi2two} and \ref{lem:pi2one} for the even period case. The proofs are postponed to Section \ref{sec:pi2lemmas}.
\begin{lem}\label{lem:pi2odd1}
When $\ell$ is odd, the numbers $R_t, t=1,\ldots, q_\ell -1$ defined in \eqref{Rtdefinition} satisfy 
\begin{equation} \label{eq:boundRtplusb}
\prod_{t=1}^{q_\ell-1} \|R_t+b\| \geq  \frac{4\pi}{ e^3 (2e)^{q_\ell}} \, \cdot 
\end{equation}
\end{lem}
\begin{lem}\label{lem:pi2odd2}
The product $\Pi_2'$ defined in \eqref{eq:pi23marked} is bounded below by
\begin{equation}\label{eq:boundpi2marked}
    \Pi_2' \geq 2^{c-\lfloor \frac{a_k}{2} \rfloor-1} \cdot\left\lfloor \frac{a_k}{ 2}\right\rfloor ! \cdot [(a_k-1)!]^{q_\ell-1} \cdot \prod_{t=1}^{q_\ell-1}\|R_t+b\| 
\end{equation}
\end{lem}
Let us now bound $\Pi_2'$ using Lemmas \ref{lem:pi2odd1} and \ref{lem:pi2odd2}. We bound the factorials in \eqref{eq:boundpi2marked} using \eqref{factorials}, and combined with Lemma \ref{lem:pi2odd1} this gives
\begin{equation}\label{eq:pi2markedfinal}
\Pi_2' \geq \frac{3\sqrt{a_k}}{20} \cdot 2^{c-a_k}  \left(\frac{1}{e\sqrt{a_k}} \right)^{q_{\ell}} \cdot  \Big(\frac{a_k}{e}\Big)^{a_k(q_\ell -\frac12)} \cdot 
\end{equation}
Combining the bounds for $\Pi_1$, $\Pi_2'$ and $\Pi_3'$ established in \eqref{eq:Pi1bd}, \eqref{Pi_3_odd_estimate} and \eqref{eq:pi2markedfinal}, we get the following analogue of Lemma \ref{lem:gprodlower}. The proof is omitted. 
\begin{lem}\label{lem:Gboundodd}
Let $G$ be defined in \eqref{gdef}. Under the assumption that $a_k\geq 6$ we have 
\begin{equation*}
    \prod_{s=0}^{c-1} |G(2s+1-2b)| \geq \left( \frac{2a_k}{\pi e^{f(a_k)} c}\right)^c \cdot \frac{3\pi \cdot 2^{-a_k}}{20e^{19/5}q_{\ell}^2} \cdot \left( \frac{e}{a_k^{5/2}}\right)^{q_{\ell}}.
\end{equation*}
\end{lem}

With Lemma \ref{lem:Gboundodd} established, we are equipped to prove Theorem \ref{thm:boundCkodd}.
\begin{proof}[Proof of Theorem \ref{thm:boundCkodd}]
Recall from \eqref{cest2} that 
\begin{equation*}
    C_k^c \leq \left( \prod_{s=1}^c |s-a| \prod_{s=0}^{c-1} |G(2a-2s-1)|\right)^{-1},
\end{equation*}
where we have used that $|G(1)|\leq 1$. Using \eqref{sminusa} and Lemma \ref{lem:Gboundodd}, we get
\begin{equation*}
    C_k^c \leq 40 c^{\frac32} \left( \frac{\pi e^{1+f(a_k)}}{2a_k}\right)^c \cdot 2^{a_k} \cdot q_{\ell}^2 \left( \frac{a_k^{\frac52}}{e}\right)^{q_{\ell}},
\end{equation*}
where we have replaced $(c+1/13)$ in \eqref{sminusa} by the upper bound $27c/26$. Raising both sides to the power $1/c$ and recalling that $c>a_kq_{\ell}$, we get
\begin{align*}
C_k &\leq \frac{\pi}{2a_k}e^{1+f(a_k)} \left(40 c^{\frac32} \right)^{\frac{1 }{c}} \cdot 2^{\frac{1 }{q_{\ell}}} \cdot q_\ell^{\frac{2}{a_k q_\ell}} \left( \frac{a_k^{ \frac{5}{2}}}{e}\right)^{\frac{1 }{a_k}} \\
&\leq \frac{\pi}{2a_k}e^{1+f(a_k)} \left(40 c^{\frac32} \right)^{ \frac{1 }{ c}} \cdot 2^{\frac{1 }{ q_{\ell}}} \cdot a_k^{\frac{5 }{2a_k}}, 
\end{align*}
where in the final step we have used that $q_{\ell}^{2/q_{\ell}}<e$.
\end{proof}

\subsection{Estimation of $\Pi_2$ and $\Pi_2'$\label{sec:pi2lemmas}}
Let us now turn to the proofs of Lemmas \ref{lem:pi2two} -- \ref{lem:pi2one} and \ref{lem:pi2odd1} -- \ref{lem:pi2odd2}, which lay the 
foundation for the estimates of $\Pi_2$ and $\Pi_2'$ above. 

\subsubsection{Proofs of Lemmas \ref{lem:pi2two} and \ref{lem:pi2odd1}}
We consider first Lemmas \ref{lem:pi2two} and \ref{lem:pi2odd1}, which provide bounds on products of factors $\|R_t\|$ and $\|R_t+b\|$, with $R_t$ given in \eqref{Rtdefinition}. Recall that these are analogous statements relevant to the cases of even and odd period lengths $\ell$, respectively. 

We treat first the proof of Lemma \ref{lem:pi2two}, and begin by examining the size of each of the three terms appearing in the expression for $R_t$ in \eqref{Rtdefinition}.
\begin{lem} \label{specialvaluelowerbound}
	When $\ell$ is even, we have
	\begin{equation*} 
	-\frac{1}{q_\ell q_{\ell+1}}\, < \, \left(\frac{p_\ell}{q_\ell} -\alpha_{\sigma_k}\right) -\frac{2b}{q_\ell}\, <\,  - \frac{a_k}{a_k +1}\, \frac{1}{q_\ell q_{\ell+1} }   \, \,  \cdot
	\end{equation*}
\end{lem}
\begin{proof}
First we observe that setting $m=1$ and $k=0$ in \eqref{ntherror} yields  $p_\ell - q_\ell \alpha_{\sigma_{k}} = b $, and therefore 
\begin{equation} \label{ltherror}
  \left(\frac{p_\ell}{q_\ell} -\alpha_{\sigma_k}\right)\, -\, \frac{2b}{q_\ell} \,\, =\,\,- \frac{b }{q_\ell} \, \cdot 
\end{equation}
When $\ell$ is even, the definition of $b$  in \eqref{abc}  gives
	\begin{eqnarray*}
b &<& \frac{1}{c-1 } \, =  \, \frac{1}{q_{\ell+1} + p_\ell -1} \, < \, \frac{1}{q_{\ell+1} }
	\end{eqnarray*}		
	and
	\begin{eqnarray*}
 b \, &>&\, \frac{1}{c} \, \,= \,\, \frac{1}{q_{\ell+1} + p_\ell }	\,\,= \, \, \frac{1}{ q_{\ell+1}}\left( 1 + \frac{p_\ell }{q_{\ell+1}}\right)^{-1}\\
&\, >\,& \frac{1}{q_{\ell+1}}\left( 1 + \frac{q_\ell }{q_{\ell+1}}\right)^{-1} \,\, > \,\, \frac{1 }{q_{\ell+1}}\left( 1+ \frac{q_\ell }{a_k q_\ell} \right)^{-1} \,=\, \frac{1}{ q_{\ell+1}}  \,\,\frac{a_k}{a_k +1} \,\, ,
	\end{eqnarray*}	
whence the claim follows. 
\end{proof}
\begin{lem} \label{fractionalpartlowerbound}
	Let $\ell$ be even. For each $t=1,2,\ldots, q_{\ell} -1$  there exists a unique integer $i=i_t\in \{1,2,\ldots, q_\ell -1\}$ such that 
	\begin{equation*}    
	\frac{i}{q_\ell} - \frac{1 }{  q_{\ell+1} }  \, \, \leq \, \,\,  \left\{ \frac{tp_\ell}{ q_\ell} \right\} + t\left(\frac{p_\ell}{ q_\ell} -\alpha_{\sigma_k} \right) -\frac{2bt}{ q_\ell} \, \,\, \leq  \,\, \frac{i}{q_\ell}   - \frac{6}{7 q_\ell q_{\ell+1}}   \,\,  \cdot
	\end{equation*}
Moreover, the correspondence $t \mapsto i_t$ is one-to-one. 
\end{lem}
\begin{proof}   
In view of the condition $a_k \geq 6$, by Lemma \ref{specialvaluelowerbound} we get 
	$$ -\frac{1}{q_{\ell+1}} \, <\, t\left(\frac{p_\ell}{ q_\ell} -\alpha_{\sigma_k} \right) -\frac{2bt}{q_\ell}  <\, - \frac{6 }{ 7q_\ell q_{\ell+1}} , \quad t=1,2,\ldots, q_\ell-1. $$
	Also since $(p_{\ell}, q_{\ell} ) = 1$, for each $t=1,2,\ldots, q_\ell-1$ there exists a unique integer $1\leq i < q_\ell$ such that $\left\{{t p_\ell / q_\ell} \right\} = {i/ q_\ell} ,$ and the correspondence $t\mapsto i$ is one-to-one. The result follows. 
\end{proof}
Lemma \ref{fractionalpartlowerbound} gives the following estimates for the values of $\|R_t\|, \, t=1,2,\ldots, q_\ell-1.$ 
\begin{cor} \label{corollary}
	Let $\ell$ be even  and $R_t $ be as in \eqref{Rtdefinition}. Then for every $t=1,2,\ldots, q_\ell-1$ there exists an integer $j_t \in \left\{  1,2,\ldots, \lfloor\tfrac12 q_\ell\rfloor  \right\} $ such that 
\begin{equation} \label{rtdistance}
\|R_t\| \, \geq \, \frac{j_t}{ q_\ell}  - \frac{1 }{q_{\ell+1}}  \quad \text{ or } \quad \|R_t\| \, \geq \, \frac{j_t}{ q_\ell}  +\frac{6 }{7 q_\ell q_{\ell+1}} \, \cdot
\end{equation}
The first inequality in \eqref{rtdistance} holds when the residue of $tp_\ell(\alpha_{\tau_k})$ modulo $q_\ell(\alpha_{\sigma_k})$ is one of $1,2,\ldots, \lfloor \tfrac12 q_\ell(\alpha_{\tau_k} ) \rfloor$ while the second holds when the residue of $tp_\ell(\alpha_{\tau_k})$ modulo $q_\ell(\alpha_{\sigma_k})$ is one of  $\lfloor \tfrac12 q_\ell(\alpha_{\tau_k}) \rfloor+1, \ldots, q_\ell(\alpha_{\tau_k}) -1 .$ Moreover, there exists at most one integer $t$ such that $j_t=\lfloor \tfrac12 q_\ell(\alpha_{\tau_k} ) \rfloor$, and  the correspondence $t\mapsto j_t$ is ``at most two-to-one". 
\end{cor}
With Corollary \ref{corollary} established, we are equipped to complete the proof of Lemma \ref{lem:pi2two}.
\begin{proof}[Proof of Lemma \ref{lem:pi2two}]
We consider first the case $q_\ell \geq 4$ and for abbreviation set $Q=\lfloor \frac12 q_\ell \rfloor$. From Corollary \ref{corollary} it follows that 
\begin{eqnarray}
\prod_{t=1}^{q_\ell -1}  \|R_t\|  & \geq &  \left(\frac12 - \frac{1}{q_{\ell+1}}  \right)\prod_{j=1}^{Q-1}\!\!  \left( \frac{j}{q_\ell}  - \frac{1 }{q_{\ell+1}} \right)    \left( \frac{j }{q_\ell} + \frac{6 }{ 7 q_\ell q_{\ell+1}} \right)     \nonumber \\
&\geq & \left(\frac{1}{ 2} - \frac16 \right) \frac{1}{q_{\ell}}\left( 1-\frac{ q_{\ell}}{ q_{\ell+1}}\right)  \prod_{j=2}^{Q-1}\left( \frac{j}{q_{\ell}} -\frac{1}{q_{\ell+1}}\right) \prod_{j=1}^{Q-1} \frac{j}{q_{\ell}} \nonumber  \\
&\geq &  \frac{5}{18q_{\ell}} \prod_{j=1}^{Q-2}\left( \frac{j}{q_{\ell}} + \frac{1}{2 q_{\ell}} \right) \prod_{j=1}^{Q-1} \frac{j}{q_{\ell}}  \qquad \left(\text{because }  1-\frac{ q_{\ell}}{ q_{\ell+1}} > \frac56  \right).    \nonumber   
\end{eqnarray}
We may therefore continue to find
\begin{eqnarray*}
\prod_{t=1}^{q_\ell -1}  \|R_t\|  &\geq& \frac{5}{ 18}\frac{(2Q-3)!!}{(2q_{\ell})^{Q-1}}\frac{(Q-1)!}{(q_{\ell})^{Q-1}} = \frac{5}{18} \frac{(2Q-2)!}{ (2q_{\ell})^{2(Q-1)}} \\
& \stackrel{\eqref{factorials}}  \geq& \frac{5}{18} \sqrt{ 2\pi \cdot 2(Q-1)} \left(\frac{Q-1}{q_{\ell}e}\right)^{2(Q-1)} \\
&  \geq &    \frac{5}{18} \sqrt{2\pi (q_{\ell}-3)} \left(\frac{q_{\ell}-3}{2q_{\ell}e}\right)^{q_{\ell}- 2 } \\
&\geq &  \frac{5}{18} \sqrt{2\pi(q_{\ell}-3)} \left(\frac{1}{2e}\right)^{q_{\ell}-2} \left(1-\frac{3}{q_{\ell}}\right)^{q_{\ell}-3}   \left(  1 - \frac{3}{ q_\ell} \right) \\
&\geq& \frac{5}{18}\,  \frac{ \sqrt{2\pi q_\ell} }{(2e)^{q_\ell+1}  } \cdot 
\end{eqnarray*}	
In the last step we have used the bound $\tfrac{q_\ell-3}{q_\ell} \geq \tfrac14$ as well as the inequality $\left(1+\frac{r}{n}\right)^{n}<e^{r}$ for $n>r$, which implies that $  \left(1-\frac{r}{n}\right)^{n-r}=\left( \frac{n}{n-r}\right)^{-(n-r)}   > e^{-r}$  whenever $n>r$. \par 
We have now shown that \eqref{eq:boundRt} holds for $q_\ell \geq 4$, and proceed by considering smaller values of $q_\ell$. When $q_\ell=1$, the product in question is empty and hence by definition equals $1$. When $q_\ell=2$, the product consists only of the factor $\|R_1\|$; by \eqref{rtdistance} and the assumption $a_k\geq 6$ we have $\|R_1\| \geq 1/3,$ so \eqref{eq:boundRt} is still valid. Finally, when $q_\ell=3$  again \eqref{rtdistance} implies $\|R_1\|\|R_2\|\geq 1/12$ which proves that \eqref{eq:boundRt} is true also in this case. This completes the proof of Lemma \ref{lem:pi2two}.
\end{proof}

We now turn to Lemma \ref{lem:pi2odd1}. We will not write out the proof of this result in full detail, as it is very similar to that of Lemma \ref{lem:pi2two}. We simply point out that it is a consequence of the following results, analogous to Lemmas \ref{specialvaluelowerbound}--\ref{fractionalpartlowerbound} and Corollary \ref{corollary} above.
 \begin{lem}
When $\ell$ is odd, we have 
\begin{equation} \label{64}
 \frac{a_k }{a_k +1}  \frac{1 }{  q_\ell q_{\ell+1}}\, <\, \left( \frac{p_\ell}{ q_\ell} -\alpha_{\sigma_k}  \right) - \frac{2b }{ q_\ell}\, <\,  \frac{1 }{  q_\ell q_{\ell+1}}\, \cdot 
\end{equation}
\end{lem}

\begin{lem} \label{fractionalpartlowerboundodd}
	Let $\ell$ be odd. For each $t=1,2,\ldots, q_{\ell} -1$  there exists a unique integer $i=i_t\in \{1,2,\ldots, q_\ell -1\}$ such that 
	\begin{equation*} 	   
	\frac{i}{q_\ell} - \frac{1 }{ q_{\ell+1} }  \, \, \leq \, \,\, \| R_t +b\| \, \,\, \leq  \,\,   \frac{i }{q_\ell}  + \frac{1}{7q_{\ell+1}}   \,\,  \cdot  
	\end{equation*}
	Moreover, the correspondence $t \mapsto i_t$ is one-to-one. 
\end{lem}

\begin{cor} \label{corollary2}
	Let $\ell$ be odd and $R_t$ be as in \eqref{Rtdefinition}. Then for every $t=1,2,\ldots, q_\ell-1$ there exists an integer $j_t \in \left\{  1,2,\ldots, \lfloor\tfrac12 q_\ell\rfloor \right\} $ such that 
	\begin{equation} \label{Stdistance}
	\|R_t+b\| \, \geq \, \frac{j_t}{q_\ell}  - \frac{1 }{ q_{\ell+1}} \quad \text{ or } \quad 	\|R_t+b\| \, \geq \, \frac{j_t}{ q_\ell}  - \frac{1}{7 q_{\ell+1}} \,\, \cdot	
\end{equation}
 The first inequality in \eqref{Stdistance} holds when the residue of $tp_\ell(\alpha_{\tau_k})$ modulo $q_\ell(\alpha_{\sigma_k})$ is one of $1,2,\ldots, \lfloor \tfrac12 q_\ell(\alpha_{\tau_k} ) \rfloor$ while the second holds when the residue of $tp_\ell(\alpha_{\tau_k})$ modulo $q_\ell(\alpha_{\sigma_k})$ is one of  $\lfloor \tfrac12 q_\ell(\alpha_{\tau_k}) \rfloor+1, \ldots, q_\ell(\alpha_{\tau_k}) -1 .$ Moreover, there exists at most one integer $t$ such that $j_t=\lfloor \tfrac12 q_\ell(\alpha_{\tau_k} ) \rfloor$, and  the correspondence $t\mapsto j_t$ is ``at most two-to-one".
\end{cor}

\subsubsection{Proofs of Lemmas \ref{lem:pi2one} and \ref{lem:pi2odd2}} Now let us treat Lemmas \ref{lem:pi2one} and \ref{lem:pi2odd2}. Recall that these provide lower bounds on the factors $\Pi_2$ and $\Pi_2'$, given in \eqref{product} and \eqref{eq:pi23marked}. As in the previous subsection, we will provide a full proof of Lemma \ref{lem:pi2one}, and then simply sketch the proof of Lemma \ref{lem:pi2odd2}.
\begin{proof}[Proof of Lemma \ref{lem:pi2one}]
Recall from \eqref{product} that $\Pi_2$ is defined as 
\begin{equation*}
    \Pi_2 = \prod_{s= \lfloor \frac{a_k}{2} \rfloor}^{c-1} \text{dist} \left( 2s+1, \cu \right), \qquad u_k(t) = \frac{2t}{|e_kc_k|}-2\{t\alpha_{\sigma_k}\}+1.
\end{equation*}
For each $s \geq 0$ let $t_s\geq 1$ be the unique positive integer such that 
\begin{equation}\label{eq:ts}
\mathrm{dist}(2s+1,\mathcal{U}) = |2s + 1 - u_k(t_s)|. 
\end{equation}
Then 
\begin{eqnarray} \label{prod2}
\Pi_2\hspace{-1mm} &=&\hspace{-5mm}  \prod_{s= \lfloor \frac{a_k}{2} \rfloor+1}^{c-1}\!\!|2s + 1 - u_k(t_s)| \, =\,  2^{c- \lfloor \frac{a_k}{2}\rfloor-1 }\, \prod_{t=1}^{\infty} \mathop{\prod_{s= \lfloor \frac{a_k}{2} \rfloor+1}^{c-1}}_{t_s =t}\!\!\tfrac12 | u_k(t)- (2s + 1)| \, \,  .
\end{eqnarray}

We now analyze the factors appearing in \eqref{prod2}. For any $t\geq 1$ and $s=0,1,\ldots, c-1$,
\begin{eqnarray} 
\frac{1}{2}\left(u_k(t) -(2s+1)\right) &=&   \frac{t}{\ckek} - \{t\alpha_{\sigma_k}\} -s  \nonumber  \\
& \stackrel{\eqref{ckekformula} }{= } &   ta_k + t\left( \frac{p_{\ell}(\alpha_{\sigma_k})}{q_{\ell}(\alpha_{\sigma_k})} + \frac{p_{\ell}(\alpha_{\tau_k})}{q_{\ell}(\alpha_{\tau_k})} - \frac{2b}{q_{\ell}(\alpha_{\tau_k})} \right) - \{t\alpha_{\sigma_k}\} -s  \nonumber  \\
&=& \underbrace{ ta_k + \left\lfloor \frac{t p_{\ell}(\alpha_{\tau_k})}{q_{\ell}(\alpha_{\tau_k})} \right\rfloor + \lfloor t\alpha_{\sigma_k}\rfloor  -s}_\text{$\in \mathbb{Z} $ }\,\, +\,\,  R_t,  \label{explicitdistance} 
\end{eqnarray}
where the terms $R_t$ are as in \eqref{Rtdefinition}, with this definition extended to all values $t\geq 1.$ 
When $t=q_\ell $ and $s=c-1$ the right hand side of \eqref{explicitdistance} is equal to 
\begin{multline*} 
a_kq_\ell + p_\ell + \lfloor q_\ell \alpha_{\sigma_k} \rfloor - (c-1) +q_\ell\left(\frac{p_\ell}{q_\ell} -\alpha_{\sigma_k}\right) -2b  = \\
= a_k q_\ell + q_{\ell-1} + \left\lfloor q_\ell \left(\alpha_{\sigma_k}-\frac{p_\ell}{q_\ell }\right) + p_\ell \right\rfloor - (c-1) +q_\ell\left(\frac{p_\ell}{q_\ell} -\alpha_{\sigma_k}\right) -2b =\\
=   q_\ell\left(\frac{p_\ell}{q_\ell} -\alpha_{\sigma_k}\right) -2b  \, \stackrel{\eqref{ltherror}}{= }  -b  \in \left( -\frac{1}{q_{\ell+1}} , 0 \right) . 
\end{multline*}
This observation implies that $t_{c-1}=q_\ell$.
Since $t_s$ is increasing in $s$, $t=q_{\ell}$ is the maximum value of $t$ for which the product in \eqref{prod2} is non-empty, and we may write
\begin{equation} \label{prod3}
\Pi_2 = 2^{c- \lfloor \frac{a_k}{2}\rfloor -1 }\,  \prod_{t=1}^{q_\ell}  \mathop{\prod_{s= \lfloor \frac{a_k}{2} \rfloor+1}^{c-1}}_{t_s =t}    \tfrac12 | u_k(t)- (2s+1)|.
\end{equation}
We now want to find lower bounds for the factors in the innermost product in \eqref{prod3} for each $1\leq t\leq q_\ell.$
When $t=1$, the values of $s$ appearing in the product 
\begin{equation*} 
\mathop{\prod_{s= \lfloor \frac{a_k}{2} \rfloor+1}^{c-1}}_{t_s =1}  \!\tfrac12 | u_k(1)- (2s+1)|
\end{equation*} 
are the integers $s\geq 1 + \lfloor \frac{a_k }{ 2} \rfloor$ for which $2s+1 < \tfrac12 \left( u_k(1) + u_k(2)\right),$ or equivalently
$$ 1 + \left\lfloor \frac{a_k }{ 2} \right\rfloor  \, \leq\, s \, <\, \frac{3a_k }{ 2} + \frac32 \left( \frac{\pl}{ \ql} + \frac{\pltau }{ \qltau} - \frac{2b }{\qltau}   \right) -  \frac{\{\alpha_{\sigma_k}\} + \{2\alpha_{\sigma_k}\} }{ 2} \, \cdot   $$
 The number of integers $s$ in any interval of the form $[n,\kappa)$, where $n$ is an integer, is $1+ \lfloor \kappa \rfloor -n $. In our case we have $n=1 + \lfloor \frac{a_k}{2} \rfloor$ and 
\begin{eqnarray*}
	\kappa &=& \frac{3a_k}{2} + \frac{3}{2} \left( \frac{\pl }{ \ql} + \frac{\pltau }{\qltau} -\frac{2b }{\qltau} \right) - \frac{\{\alpha_{\sigma_k}\} + \{2\alpha_{\sigma_k}\} }{2} \\
	& \geq & \frac{3a_k}{2} + \frac{3}{2} \left( \frac{\pl }{ \ql} + \frac{\pltau }{\qltau} -\frac{2b }{\qltau}   \right) -    \frac32 {\{\alpha_{\sigma_k} \} } \\
	&=&  \frac{3a_k}{2} + \frac{3}{2} \Big(   \underbrace{ \frac{\pl}{\ql}  -\alpha_{\sigma_k} }_{ \frac{1}{ 2q_\ell q_{\ell+1}} \leq \ldots \leq \frac{1}{q_\ell q_{\ell+1} }}  + \underbrace{ \frac{\pltau }{ \qltau} -\frac{2b }{ \qltau} }_{>0 }  \Big),
\end{eqnarray*}	
so there are at least $a_k$ values of $s$ for which $t_s=1$. Consider now the  corresponding factors in the innermost product in \eqref{prod3}. From \eqref{explicitdistance} and Lemma  \ref{fractionalpartlowerbound} it follows that the minimal factor is equal to $\|R_1\|$, whereas the remaining factors can be bounded below by the values $ 1,2,\ldots, a_k-1 .$ Therefore
\begin{equation} \label{p1}
\mathop{\prod_{s= \lfloor \frac{a_k}{2} \rfloor+1}^{c-1}}_{t_s =1 }\!\!\tfrac12 | u_k(1)- (2s+1)| \, \geq\, \|R_1\| \cdot   (a_k-1)    !
\end{equation}
We then examine the innermost product of \eqref{prod3} for $1< t < q_\ell$. The factors appearing are those corresponding to integers $s$ such that 
$$ \tfrac12\left( u_k(t-1)+ u_k(t) \right) \, < \, 2s + 1 \, < \, \tfrac12 \left(u_k(t) + u_k(t+1) \right), $$
or equivalently 
$$ \frac{2t-1}{|c_ke_k| }  -\{(t-1)\alpha_{\sigma_k}\} -\{t\alpha_{\sigma_k}\} +1 < 2s+1 < \frac{2t+1}{|c_ke_k| } -\{t\alpha_{\sigma_k}\} - \{(t+1)\alpha_{\sigma_k}\} +1 .  $$
Observe that the number of integers in an interval $(\alpha, \beta)$ is at least $\lfloor \beta-\alpha\rfloor.$ Here the length of the interval of possible values of $s$ is 
\begin{eqnarray*}
	\left( \frac{t +\frac12 }{\ckek}  - \frac{\{t\alpha_{\sigma_k}\} -\{(t+1)\alpha_{\sigma_k} \} }{2} \right) & - &  \left( \frac{t - \frac12 }{\ckek}  - \frac{\{t\alpha_{\sigma_k}\} -\{(t-1)\alpha_{\sigma_k} \} }{2 }   \right) =  \\
	& =& \frac{1}{ |c_ke_k| } + \frac{ \{(t-1)\alpha_{\sigma_k}\}  + \{(t+1)\alpha_{\sigma_k} \} }{2} \\
	& \geq & a_k + \frac{\pl}{\ql} + \frac{\pltau }{\qltau} -\frac{2b }{\qltau}  -\alpha_{\sigma_k} \\
	& \geq & a_k,
\end{eqnarray*}	
so there exist at least $a_k$ values of $s$. The minimal factor $\frac{1}{2}|u_k(t)-(2s+1)|$ is  equal to $\|R_t\|$, while the remaining ones can be bounded from below by $1,2,\ldots, a_k -1.  $
We thus have
\begin{equation} \label{p2}
 \mathop{\prod_{s= \lfloor \frac{a_k}{2} \rfloor+1}^{c-1}}_{t_s=t}\!\!\tfrac12 | u_k(t)- (2s+1)|\, \geq\, \|R_t\|  \cdot (a_k - 1)!,\, \qquad t=2,\ldots, q_\ell -1 .
\end{equation}

Finally we deal with the innermost product in \eqref{prod3} when $t=q_\ell$. The values of $s $ appearing in the product are those for which 
\begin{eqnarray*}
2c-1\, \geq\, 2s+1 &>& \tfrac12 \left( u_k(q_\ell) +u_k(q_\ell-1) \right) \\
	&=& 2(c-2b)-\left(  a_k + \frac{p_{\ell}(\alpha_{\sigma_k})}{q_{\ell}(\alpha_{\sigma_k})} + \frac{p_{\ell}(\alpha_{\tau_k})}{q_{\ell}(\alpha_{\tau_k})} - \frac{2b}{q_{\ell}(\alpha_{\tau_k})}\,  \right)\\
	& \, & \qquad \qquad -\{q_\ell \alpha_{\sigma_k}\} - \{(q_\ell-1)\alpha_{\sigma_k}\} + 1 ,
\end{eqnarray*}
which is equivalent to 
\begin{eqnarray*}
	c-1 \,\,  \geq \,  \,  s &>&  c - \frac12 \left(  a_k + \frac{p_{\ell}(\alpha_{\sigma_k})}{q_{\ell}(\alpha_{\sigma_k})} + \frac{p_{\ell}(\alpha_{\tau_k})}{q_{\ell}(\alpha_{\tau_k})} - \frac{2b}{q_{\ell}(\alpha_{\tau_k})}\,   \right) \\
	& & \qquad  - \frac{ \{q_\ell \alpha_{\sigma_k}\} + \{ (q_\ell-1)\alpha_{\sigma_k}\}}{2 }  - 2b \,. 
\end{eqnarray*} 
For any $\kappa >1$, the interval $(c-\kappa, c-1]$ contains precisely $\lfloor\kappa \rfloor$ integers. Here we need to apply this observation with 
\begin{eqnarray*}
	\kappa &=& \frac{1}{2} \left(  a_k + \frac{p_{\ell}(\alpha_{\sigma_k})}{q_{\ell}(\alpha_{\sigma_k})} + \frac{p_{\ell}(\alpha_{\tau_k})}{q_{\ell}(\alpha_{\tau_k})} - \frac{2b}{q_{\ell}(\alpha_{\tau_k})}\,   \right) + \frac{ \{q_\ell \alpha_{\sigma_k}\} + \{ (q_\ell-1)\alpha_{\sigma_k} \}  }{2} + 2b \\
	& \geq &  \frac{1}{2} \left(  a_k + \frac{p_{\ell}(\alpha_{\sigma_k})}{q_{\ell}(\alpha_{\sigma_k})} + \frac{p_{\ell}(\alpha_{\tau_k})}{q_{\ell}(\alpha_{\tau_k})} - \frac{2b}{q_{\ell}(\alpha_{\tau_k})}\,   \right) +
	\frac{ \left(1 - \frac{1}{q_{\ell+1}} \right) + \left(1 - \frac{1}{ q_{\ell+1}}  -\alpha_{\sigma_k} \right)    }{ 2}  + 2b  \\
	& \geq & \frac{1}{2} \left(  a_k + \frac{p_{\ell}(\alpha_{\sigma_k})}{q_{\ell}(\alpha_{\sigma_k})} - \alpha_{\sigma_k} + \frac{p_{\ell}(\alpha_{\tau_k})}{q_{\ell}(\alpha_{\tau_k})} - \frac{2b}{q_{\ell}(\alpha_{\tau_k})}\,  + 2 - \frac{2}{q_{\ell+1}} + 4b  \right) 
\end{eqnarray*}
hence there must be at least $1+ \left\lfloor \dfrac{a_k}{2} \right\rfloor $  such factors.

 When $t=q_\ell$ and $s=c-1$, we saw already that 
\begin{align*}
	\tfrac12 |u_k(q_{\ell}) - 2c+1| &=|q_{\ell}a_k + p_{\ell}(\alpha_{\tau_k}) +p_{\ell}(\alpha_{\sigma_k})-c + R_{q_{\ell}}|   =   |R_{q_{\ell}}| =b.
\end{align*} 
For the remaining values of $s$ with $t_s = q_\ell$, we bound the factors $ \tfrac12 |u_k(q_\ell)-(2s+1)|$ from below by  
$$r \, - \,  \frac{1 }{q_{\ell+1}}, \qquad    r = 1, 2, \ldots,\left\lfloor\frac{a_k}{ 2}\right\rfloor.$$  
Consequently,  
\begin{eqnarray}
\mathop{ \prod_{\lfloor a_k/2 \rfloor + 1\leq s < c -1 }}_{t_s =q_\ell}\hspace{-5mm}\tfrac12 |u_k(q_\ell)- (2s-1)|   & \geq &   b \prod_{r=1}^{\lfloor a_k /2 \rfloor}\!\!\left( r - \frac{1 }{q_{\ell+1}} \right) \nonumber \\[-2ex]
& =&    b  \left\lfloor \frac{a_k}{2} \right\rfloor !   \prod_{r =1}^{\lfloor a_k /2 \rfloor  }\!\!\left( 1 - \frac{1 }{r q_{\ell+1}}\right) \nonumber \\
& \geq &    \frac{1}{c} \left\lfloor \frac{a_k}{2} \right\rfloor ! \Big(   1- \sum_{r=1}^{\lfloor a_k/2 \rfloor   }\hspace{-2mm} \frac{1 }{r q_{\ell+1}}   \Big)    \nonumber  \\
&\geq&   \frac{4}{5c}  \left\lfloor \frac{a_k}{2} \right\rfloor   !  \quad . \label{p3}
\end{eqnarray}
On combining \eqref{prod3} with \eqref{p1}, \eqref{p2} and \eqref{p3}, we obtain inequality \eqref{eq:firstpi2bd}. This completes the proof of Lemma \ref{lem:pi2one}.
\end{proof}

Let us now give an outline of the proof of Lemma \ref{lem:pi2odd2}. Recall that this result gives a bound on $\Pi_2'$ defined in \eqref{eq:pi23marked}, relevant to the case of odd period lengths $\ell$.
Using the same arguments that gave the bound for $\Pi_2$ above, we find that 
 \begin{equation*} 
\Pi_2' = 2^{c- \lfloor \frac{a_k}{2}\rfloor -1 }\,  \prod_{t=1}^{q_\ell}  \mathop{\prod_{s= \lfloor \frac{a_k}{2} \rfloor+1}^{c-1}}_{t_s =t}    \tfrac12 | u_k(t)- (2s+1-2b)|,
\end{equation*}
where $t_s$ is defined in \eqref{eq:ts}.
We now proceed as before, and find lower bounds on the factors in the innermost product for each $1 \leq t \leq q_{\ell}$. 
According to \eqref{explicitdistance}, the factors appearing in the product are
\begin{eqnarray} \label{oddFactors}
\frac{1}{2}\left(u_k(t) -(2s+1-2b)\right) &=& \underbrace{ ta_k + \left\lfloor \frac{tp_{\ell}(\alpha_{\tau_k})}{q_{\ell}(\alpha_{\tau_k})} \right\rfloor + \lfloor t\alpha_{\sigma_k}\rfloor  -s}_\text{$\in \mathbb{Z} $ }\,\, +\,\,  R_t + b,   
\end{eqnarray}
with $R_t$ given in \eqref{Rtdefinition}. 

When $\ell$ is odd, $t=q_\ell$ and $s=c-1$, the right hand side of \eqref{oddFactors} is equal to 
\begin{multline} 
a_kq_\ell + p_\ell + \lfloor q_\ell \alpha_{\sigma_k} \rfloor - c + 1 + q_\ell\left(\frac{p_\ell}{q_\ell} -\alpha_{\sigma_k}\right) - b  = \nonumber \\
= a_k q_\ell + q_{\ell-1} + \left\lfloor q_\ell \left(\alpha_{\sigma_k}-\frac{p_\ell }{q_\ell }\right) + p_\ell \right\rfloor - c +1 +q_\ell\left(\frac{p_\ell}{q_\ell} -\alpha_{\sigma_k}\right) -  b \nonumber \\
= q_{\ell+1} + p_\ell - (q_{\ell+1} + p_\ell)  +1 + q_\ell\left(\frac{p_\ell}{q_\ell} -\alpha_{\sigma_k}\right) - b  \nonumber   \\
= 1+  q_\ell\left(\frac{p_\ell}{q_\ell} -\alpha_{\sigma_k}\right) -   b  \, \stackrel{\eqref{ltherror}}{=}  1. \qquad \qquad \quad \label{minsize}
\end{multline}
It follows that 
\begin{equation*}  
\mathop{\prod_{s= \lfloor \frac{a_k}{2} \rfloor+1}^{c-1}}_{t_s=q_\ell}  \tfrac12 |u_k(q_\ell) - (2s+1-2b)|\, \, \geq   \, \,  \left\lfloor \frac{a_k}{ 2}\right\rfloor  ! \quad  .
\end{equation*}
For $t=1, \ldots , q_{\ell}-1$, we argue as for the even period case in Lemma \ref{lem:pi2one}, and obtain
\begin{equation*}
 \mathop{\prod_{s= \lfloor \frac{a_k}{2} \rfloor+1}^{c-1}}_{t_s=t}\!\!\tfrac12 | u_k(t)- (2s+1-2b)|\, \geq\, \|R_t+b\|  \cdot (a_k - 1)!,\, \qquad t=1,\ldots, q_\ell -1  .
\end{equation*}
Combining these two estimates, we deduce that
\begin{eqnarray*}
\Pi_2' &\geq& 2^{c- \lfloor \frac{a_k}{ 2}\rfloor -1}\cdot [(a_k-1)!]^{q_\ell-1}\cdot \left\lfloor \frac{a_k}{ 2}\right\rfloor ! \cdot \prod_{t=1}^{q_\ell-1}\|R_t+b\| ,
\end{eqnarray*}
which confirms Lemma \ref{lem:pi2odd2}.

\section{Proof of Theorem \ref{thm:ak23}\label{sec:proofthm3}}
With Theorems \ref{thm:boundCk} and \ref{thm:boundCkodd} established, let us now prove Theorem \ref{thm:ak23}. As explained in Remark \ref{rem4}, it suffices to study quadratic irrationals $\alpha$ with purely periodic continued fraction expansion. \par  Recall that by Theorem \ref{thm2}, Theorem \ref{thm:ak23} is verified if we can show that 
\begin{equation*}
    C_k = \lim_{m \to \infty} P_{q_{m \ell +k}}(\alpha) < 1
\end{equation*}
for every $\alpha=[0;\overline{a_1, a_2, \ldots, a_{\ell}}]$ with $a_k = \max_j a_j \geq 23$.

Assume first that the period length $\ell$ is odd. Then $q_{\ell}>1$, and from Theorem \ref{thm:boundCkodd} it follows that 
\begin{equation*}
    C_k \leq \frac{\pi}{\sqrt{2}a_k}e^{1+f(a_k)} \left( 40c^{\frac32}\right)^{\frac{1 }{ c}} a_k^{ \frac{5}{2a_k}}.
\end{equation*}
The factor $(40c^{3/2})^{1/c}$ is decreasing in $c$ (within the range of relevant values of $c$), so we may safely use the fact that $c > a_kq_{\ell}$ to obtain
\begin{align*}
    C_k &\leq \frac{\pi}{\sqrt{2}a_k}e^{1+f(a_k)} \left( 40a_k^{\frac32}q_{\ell}^{\frac32}\right)^{ \frac{1}{a_kq_{\ell}}} a_k^{ \frac{5 }{ 2a_k}}\\
    &\leq \frac{\pi}{\sqrt{2}a_k}e^{1+f(a_k)} \left( 40a_k^{\frac32}\right)^{ \frac{1 }{ a_kq_{\ell}}} 2^{\frac{1}{a_k}} a_k^{\frac{5 }{ 2 a_k}},
\end{align*}
where for the last inequality we have used that $q_{\ell}^{3/2q_{\ell}}<2$. This expression is decreasing in $q_{\ell}$, so we insert $q_{\ell}=2$ to obtain
\begin{equation*}
    C_k \leq \frac{\pi}{\sqrt{2}a_k}e^{1+f(a_k)} \left( 160a_k^{\frac{13}{2}}\right)^{ \frac{1}{2a_k}}.
\end{equation*}
The right hand side is a decreasing function of $a_k$, and it can be easily verified that $C_k \leq 1$ whenever $a_k\geq 22$. This proves Theorem \ref{thm:ak23} for odd period lengths $\ell$.

Now assume that the period length $\ell$ is even, and consider first the case $q_{\ell}>1$. By Theorem \ref{thm:boundCk} we then have 
\begin{equation*}
    C_{k}^{\frac{c-2 }{c}} \leq \frac{\pi}{\sqrt{2} a_k}e^{1+f(a_k)}(200e^{2.4}c^2)^{\frac{1}{c}} \left( \frac{a_k^{\frac52}}{e}\right)^{\frac{1}{a_k}}. 
\end{equation*}
The factor $(200e^{2.4}c^2)^{1/c}$ is again decreasing in $c$, so we replace $c$ by $a_kq_{\ell}$ to obtain
\begin{align*}
    C_{k}^{\frac{c-2 }{c}} &\leq \frac{\pi}{\sqrt{2}a_k}e^{1+f(a_k)}(200e^{2.4}a_k^2q_{\ell}^2)^{\frac{1 }{ a_k q_{\ell}}} \left( \frac{a_k^{\frac52}}{e}\right)^{\frac{1 }{a_k}}\\
    &\leq \frac{\pi}{\sqrt{2}a_k} e^{1+f(a_k)}(200e^{2.4}a_k^2)^{\frac{1}{a_k q_{\ell}}} \cdot a_k^{\frac{5 }{2a_k}},
\end{align*}
where we have used that $q_{\ell}^{2/ q_{\ell}}<e$. This expression is decreasing in $q_{\ell}$, so we insert $q_{\ell}=2$ to obtain
\begin{equation*}
    C_k^{\frac{c-2 }{ c}} \leq \frac{\pi}{\sqrt{2}a_k} e^{1+f(a_k)} \left(200e^{2.4}a_k^7\right)^{\frac{1}{2a_k}}.
\end{equation*}
One can again verify that the right hand side is below one whenever $a_k \geq 23$. This proves Theorem \ref{thm:ak23} for even $\ell$ in the case $q_{\ell}>1$. 

Finally, we consider the case $q_{\ell}=1$. Recall that this can only occur if $\ell=2$ and $\alpha=[0;\overline{a_1, a_2}]$, with either $a_1=1$ or $a_2=1$. By Corollary \ref{cor:boundCkq=1}, we then have the bound 
\begin{equation*}
    C_k^{\frac{c-2}{c}} \leq \frac{\pi}{a_k}e^{1+g(a_k)} \left(6.2(a_k+2)^4 \right)^{\frac{1}{a_k+2}},
\end{equation*}
where $g(a_k) \leq 3.3/a_k+0.1$ and one can again verify that the right hand side is below one whenever $a_k \geq 21$. This verifies Theorem \ref{thm:ak23} when $q_{\ell}=1$, and completes the proof.

\end{document}